\documentclass[12pt]{amsart}
\usepackage{amsmath}
\usepackage{amsfonts, amssymb, euscript, mathrsfs}
\usepackage{amsthm, upref}
\usepackage{graphicx}
\usepackage[usenames, dvipsnames]{color}

\usepackage{tikz}
\usepackage{enumerate}
\usepackage[margin=0.8in]{geometry}
\usepackage{aliascnt}
\usepackage{hyperref}
\usepackage{verbatim}
\usepackage{bm}

\allowdisplaybreaks[4]

\newcommand{\vk}{\varkappa}

\newcommand{\BR}{\mathbb{R}}

\newcommand{\SL}{\sum\limits}

\newcommand{\de}{\delta}

\newcommand{\CF}{\mathcal F}

\newcommand{\MP}{\mathbf P}

\newcommand{\CL}{\mathcal L}

\newcommand{\Oa}{\Omega}

\newcommand{\si}{\sigma}

\newcommand{\pa}{\partial}
\renewcommand{\phi}{\varphi}

\newcommand{\eps}{\varepsilon}
\newcommand{\la}{\lambda}
\newcommand{\Ra}{\Rightarrow}

\newcommand{\ol}{\overline}

\newcommand{\norm}[1]{\lVert#1\rVert}
\renewcommand{\comment}[1]{}

\newcommand{\md}{\mathrm{d}}

\DeclareMathOperator{\supp}{supp}

\DeclareMathOperator{\TV}{TV}
\DeclareMathOperator{\mes}{mes}
\DeclareMathOperator{\Int}{int}
\DeclareMathOperator{\dist}{dist}

\begin{document}

\theoremstyle{plain}
\newtheorem{thm}{Theorem}[section]
\newtheorem*{thmnonumber}{Theorem}
\newtheorem{lemma}[thm]{Lemma}
\newtheorem{prop}[thm]{Proposition}
\newtheorem{cor}[thm]{Corollary}
\newtheorem{open}[thm]{Open Problem}
\numberwithin{equation}{section}

\theoremstyle{definition}
\newtheorem{defn}{Definition}[section]
\newtheorem{asmp}{Assumption}[section]
\newtheorem{notn}{Notation}
\newtheorem{prb}{Problem}

\theoremstyle{remark}
\newtheorem{rmk}{Remark}[section]
\newtheorem{exm}{Example}
\newtheorem{clm}{Claim}

\author{Tomoyuki Ichiba and Andrey Sarantsev}

\title[Walsh Diffusions]{Convergence and Stationary Distributions\\ for Walsh Diffusions}

\address{\scriptsize Department of Statistics and Applied Probability, University of California, Santa Barbara}

\email{ichiba@pstat.ucsb.edu}

\address{\scriptsize Department of Statistics and Applied Probability, University of California, Santa Barbara}

\email{sarantsev@pstat.ucsb.edu}

\begin{abstract}
A Walsh diffusion on Euclidean space moves along each ray from the origin, as a solution to a stochastic differential equation with certain drift and diffusion coefficients, as long as it stays away from the origin. As it hits the origin, it instantaneously chooses a new direction according to a given probability law, called the spinning measure. A special  example is a real-valued diffusion with skew reflections at the origin. This process continuously (in the weak sense) depends on the spinning measure. We determine a stationary measure for such process, explore long-term convergence to this distribution and establish an explicit rate of exponential convergence. 
\end{abstract}



\date{\today. Version 48}

\maketitle

\thispagestyle{empty}

%\section*{Questions to Resolve}
%
%There are some minor remarks in this article, they are in the footnotes. There are two major questions:
%
%(a) Prove weak existence and uniqueness in law for reflected Walsh Brownian motion on the unit ball with spinning measure $\mu$. 
%
%(b) Show that, for every $\de > 0$, the quantity on $[0, T]$ of excursions at zero of reflected Brownian motion on $[0, 1]$ which have height at least $\de$ has finite expectation. 

\section{Introduction}

\subsection{Informal description of Walsh Brownian motions and Walsh diffusions} 

Fix a positive integer $d \ge 1$, a dimension of $\BR^d$ with Euclidean norm $\norm{x} := \left(x_1^2 + \ldots + x_d^2\right)^{1/2}$. Take a Borel probability measure $\,\mu\,$ on the unit sphere $\,\mathbb S := \{z \in \BR^d\mid\norm{z} = 1\}\,$. The origin in $\BR^d$ will be denoted by $\mathbf{0}$, to distinguish it from the zero on the real line. With rays $\mathcal R_{\theta} := \{s\,  \theta \in \mathbb R^{d}\mid s \ge 0\}$, $\theta \in \mathbb S$, we see $\mathbb R^{d} = \cup_{\theta \in \mathbb S} \mathcal R_{\theta}$. Take a filtered probability space $(\Oa, \CF, (\CF_t)_{t \ge 0}, \MP)$ with the filtration satisfying the usual conditions. A {\it Walsh Brownian motion} in $\,\BR^d\,$ with {\it spinning measure} $\mu$ is an adapted,  continuous stochastic process $X = (X(t), t \ge 0)$ which is informally described as follows. 

\smallskip 

Let us take a one-dimensional reflected Brownian motion $S = (S(t),\, t \ge 0)$ with values in $\BR_+ := [0, \infty)$, starting from the origin and reflected at the origin. Its sample path can be split into {\it excursions} in a measurable way.  For {\it every} excursion, choose an $\mathbb S$-valued random variable $\, \bm \theta \,$ distributed in $\mu$, independent of these variables for other excursions and of the underlying reflected Brownian motion $S$. Define the $d$-dimensional stochastic process $X(t) =  {\bm \theta}\, S(t)\,$ for {\it each} $t$ in the open interval which the corresponding excursion straddles, and $X(t) = \mathbf 0$ for all other $t$ (where $S(t) = 0$). We call $X=(X(t), t \ge 0) $ a Walsh Brownian motion in $\mathbb R^{d}$ with spinning measure $\, \mu\,$. 

\smallskip 

A particular case is the {\it skew Brownian motion}, when the measure $\mu$ is supported on two opposite polar points (North Pole and South Pole) on $\mathbb S$. In this case, the process is essentially one-dimensional. In fact, the skew Brownian motion is usually defined as an $\BR$-valued process, where, instead of North and South Pole, we choose the positive and negative half-line. For example, if the probabilities attached to the North and South Poles are equal to $1/2$, then we re-create a standard Brownian motion. Usually, Walsh Brownian motions in $\mathbb R^{2}$ are considered in the literature, however, it was originally pointed out by    Walsh in \cite{Walsh78} that the construction and the corresponding theory for the general $d$-dimensional case is basically the same as the case of $d = 2$. 

\medskip

%\smallskip 
%
% We shall show that the Walsh semimartingale is weakly continuous with respect to its skeleton $U$ and the spinning measure $\mu$. That is, consider the sequence $(U_n)_{n \ge 0}$ of nonnegative semimartingales, starting from $U_n(0) = 0$, and a sequence $(\mu_n)_{n \ge 0}$ of Borel probability measures on $\mathbb S$. Take the corresponding Walsh semimartingales $(X_n)_{n\ge 0}$. 
%We shall show that, under some additional conditions, if $U_n \Ra U_0$ in $C[0, T]$, and $\mu_n \Ra \mu_0$ in $\mathbb S$, then $X_n \Ra X_0$ in $C([0, T], \BR^d)$. 

%\smallskip 
%One can also define generalizations of Walsh Brownian motion. Consider an $\BR_+$-valued semimartingale $U = (U(t), t \ge 0)$ starting from $U(0) = 0$, instead of a reflected Brownian motion $S = (S(t), t \ge 0)$. Then the analogous construction leads to a {\it Walsh semimartingale} $X$. We say that $U$ is a {\it skeleton} for $X$.

We can also construct {\it Walsh diffusions}, given the measure $ \mu $ on $\mathbb S$ and some coefficients $g : \BR^d \to \BR$ and $\si : \BR^d \to \BR_+$. Similarly to a Walsh Brownian motion, Walsh diffusions move along the rays $\mathcal R_{\theta} $, $\theta \in \mathbb S$,  
%Consider coefficients $g : \BR^d \to \BR$ and $\si : \BR^d \to \BR_+$. 
%Take a continuous $\BR^d$-valued process $X = (X(t), t \ge 0)$ such that it moves along the ray $\mathcal R_{\theta}$. 
that is, as long as this process is on the ray $\mathcal R_{\theta}$ for a certain $\theta \in \mathbb S$, it behaves as a solution of a stochastic differential equation (SDE) on $(0, \infty)$ with drift coefficient $g(\cdot, \theta)$ and diffusion coefficient $\si^2(\cdot, \theta)$. If necessary, we stress their dependence on both $\theta$ and $r = \norm{x}$ by writing them as $g(r, \theta)$ and $\si(r, \theta)$. 
%\medskip
Note that the meaning of spinning measure $\mu$ for this Walsh diffusion is slightly different from what we describe for the Walsh Brownian motion in the above. %We do not define spinning measure $\mu$ as governing the choice of rays for a Walsh diffusion starting from zero. Rather, f
More precisely, for every subset $A \subseteq \mathbb S$, $\mu(A)$ is the share of local time accumulated at the origin corresponding to the excursions of the Walsh diffusion on rays $\mathcal R_{\theta},\, \theta \in A$ (see (\ref{eq:rays-mu}) below).  %This new definition of spinning measure is equivalent to the previous one (measure governing choice of rays) for the case when $g$ and $\si$ are independent of $\theta$; in particular, for Walsh Brownian motion. 

\medskip

Our primary purpose of study in this paper is the long-term convergence of the Walsh diffusion to the stationary measure under appropriate conditions. Along the course, we examine in detail the local behavior of the Walsh diffusions in the neighborhood of the origin, and examine the Feller continuity and some other properties of Walsh diffusions.  
%
%\medskip
%
%We also introduce a reflected Walsh diffusion. 
%This process behaves as a Walsh diffusion, except that at each ray $\mathcal R_{\theta}$, we fix a point $R(\theta)\theta$ where this process is reflected back to the origin. % processthey are reflected back  and also study the reflected diffusions in the tree topology in the subsets of $\mathbb R^{d}$. 

\subsection{Historical review} Walsh Brownian motion was introduced in \cite{Walsh78} and further studied in \cite{BPY89}, in the two-dimensional context (but the results are immediately transferred to $\BR^d$ for $d > 2$). In much of the existing literature, the support of the spinning measure $\mu$ is finite, i.e., $\supp\mu = \{\theta_1, \ldots, \theta_k\} \subset \mathbb S$, where $\theta_{1}, \ldots , \theta_{k}$ are $k$ distinct points in $\mathbb S$. In this case, Walsh Brownian motion or Walsh diffusion has effective state space $\{r\theta_i,\, i = 1, \ldots, k,\, r \ge 0\}$, and is sometimes called a {\it spider}.  Filtrations generated by Walsh Brownian motion on a spider were studied in \cite{Tsirelson, Watanabe}. A construction of Walsh Brownian motion and, more generally, a Walsh diffusion, on a spider via pinching points together was done in \cite{Evans}. It\^o's formula for Walsh Brownian motion with general spinning measure was proved in \cite{Short}. Stochastic calculus for general tree-valued diffusion processes is developed in \cite{Sheu, Wentzell, Picard}. A Dirichlet form approach was used in \cite{ChenFukushima} to construct Walsh Brownian motion with spinning measure $\mu$, and in \cite[section 7.5, section 7.6]{ChenBook} for the spider. Here, we apply this method of Dirichlet forms to find stationary measures of Walsh diffusions with general spinning measure. We also use Dirichlet forms to construct reflected Walsh diffusions. 
%\footnote{Do you want to add more?} 
Stochastic flows and harmonic functions for Walsh Brownian motion were studied in \cite{Hajri} and \cite{Patrick}, respectively. 
%\smallskip
%For our purposes, the following two recent papers are most relevant. 
Walsh semimartingales and diffusions, with arbitrary spinning measure $\mu$, were introduced in \cite{IchibaNew} and further studied in \cite{MinghanNew} with control problems. Exponential ergodicity for a related class of Markov processes with {\it random switching} was recently studied  in \cite{Switching}. 
%As a particular case, angular-independent ($g(\theta, r) \equiv g(r)$ and $\si(\theta, r) \equiv \si(r)$) Walsh diffusions were constructed. Next, in the paper \cite{MinghanNew}, Walsh diffusions were introduced for the angular-dependent case. >>

\subsection{Overview of the paper} In section 2, we introduce notation, define Walsh Brownian motions as well as Walsh diffusions, and study some of their elementary properties. Section 2, for the most part, does not contain new results; it is a review of \cite{MinghanNew}. In section 3, we construct Walsh diffusions using a Dirichlet form and the method of {\it one-point reflection} from the paper \cite{ChenFukushima} and the book \cite[section 7.5, section 7.6]{ChenBook}. This alternative construction  helps us to find a stationary measure for a Walsh diffusion (which is, however, not necessarily a stationary distribution). In section 4, we discuss continuous dependence of the law of a Walsh diffusion $X = (X(t),\, t \ge 0)$ on the spinning measure $\mu$ and the initial condition $X(0) = x \in \mathbb R^{d}$. For the case of a Walsh Brownian motion, we quantify this continuity, effectively saying that the law of a Walsh Brownian motion is a H\"older continuous function of $\mu$ %the spinning measure 
with respect to a Wasserstein distance on $\BR^d$. 

\smallskip

In section 5, we study additional properties of Walsh diffusions, which were not considered in the previous paper \cite{MinghanNew}: positivity of the transition kernel with respect to a reference measure on $\BR^d$, and Feller continuity. In section 6, we construct Lyapunov functions for Walsh diffusions to show {\it ergodicity:} existence and uniqueness of a stationary distribution $\pi$, and convergence to $\pi$ in the total variation norm (or even stronger norms) as $t \to \infty$. Under some more restrictive conditions, we also prove {\it uniform ergodicity}: exponentially fast convergence to $\pi$  as $t \to \infty$. 

\smallskip

Our main contribution is to find explicit estimates of the rate of exponential convergence for Walsh diffusions, extending the ones in \cite{LMT1996, MyOwn12} for reflected diffusions and jump-diffusions on a positive half-line. These results are then applied to a {\it non-reflected} diffusion on the whole real line. 
%(which is just a particular case of  a spider with two rays). 
We would like to stress that often, it is relatively easy to prove that a diffusion process (or a discrete-time Markov chain) on $\mathbb R^{d}$ %the real line or in a finite-dimensional Euclidean space converges 
converges to its stationary distribution exponentially fast, but difficult to find or estimate an explicit rate of exponential convergence. Some partial results in this direction are provided in the papers~\cite{BCG2008, Davies,  Explicit, RR1996, RT1999, RT2000}. 

%\newpage 

\smallskip

%Finally, in section 7 we introduce reflected Walsh diffusions in a restricted space. We cut each ray $\mathcal R_{\theta} $ for $\theta \in \mathbb S$ at some point $R(\theta)\cdot \theta $ for some function $R: \mathbb S \to \mathbb R_{+}$. Then we get a subinterval $[\mathbf{0}, R(\theta)\theta]$ of $\mathcal R_{\theta}$ for each $\theta \in \mathbb S$. On this ray $\mathcal R_{\theta}$ this reflected Walsh diffusion behaves as a reflected diffusion on $[0, R(\theta)]$, as long as it does not hit the origin. We extend the results from \cite{MinghanNew} and from the previous sections of the current paper, for such reflected Walsh diffusions. %\smallskip
%

%the efficient coupling of two Walsh Brownian motions with different spinning measures $\mu$, $ \overline{\mu}	$, and then estimate the Wasserstein distance between a couple of two Walsh Brownian motions in Theorem \ref{Thm2.1}. This estimate guarantees the continuous dependence of distribution of Walsh Brownian motions on the spinning measure. As its corollary,  we prove Theorem \ref{Thm1} of  weak convergence results of the Walsh Brownian motion. 

%\newpage 

\section{Background and Definitions} 

\subsection{Notation} Recall that in $\BR^d$, the Euclidean norm is defined by $\norm{x} := (x_1^2 + \cdots + x_d^2)^{1/2}$ for $x = (x_1, \ldots, x_d)$. We shall denote the origin in $\BR^d$ by $\mathbf{0} = (0, \ldots, 0)$. Let $\mathbb S := \{x \in \BR^d\mid \norm{x} = 1\}$ and $\mathbb B := \{x \in \BR^d\mid \norm{x} \le 1\}$ be the unit sphere and the unit ball in $\BR^d$. For every $x \in \BR^d\setminus\{\mathbf{0}\}$, we write $x = (r, \theta)$ or simply $x = r\theta$ if $r = \norm{x} > 0$ and $\theta = \arg(x) := x/r \in \mathbb S$ (polar coordinates). We denote by $\, (x)_{-}\,$ the non-positive part of a real number $\, x \in \mathbb R\,$. We define the {\it tree-metric} as follows: for every $x_1, x_2 \in \BR^d$ with $r_i := \norm{x_i}$, $i = 1, 2$, 
\begin{equation}
\label{eq:tree}
\dist(x_1, x_2) := 
\begin{cases}
r_1,\ \ \mbox{if}\ \ r_2 = 0;\\
r_2,\ \ \mbox{if}\ \ r_1 = 0;\\
|r_1 - r_2|,\ \ \mbox{if}\ \ r_1r_2 \ne 0\ \ \mbox{and}\ \ \arg(x_1) = \arg(x_2);\\
r_1 + r_2,\ \ \mbox{if}\ \ r_1r_2 \ne 0\ \ \mbox{and}\ \ \arg(x_1) \ne \arg(x_2).
\end{cases}
\end{equation}
This essentially means that we have the usual Euclidean distance on each ray 
%\begin{equation}
%\label{eq:ray}
$\mathcal R_{\theta} $ %:= \{r\theta \in \BR^d\mid r \ge 0\},
%\end{equation}
but a continuous movement cannot jump between rays, except through the origin. One can think of it as railroads converging to the central city; this is why it is sometimes also called the {\it railway metric}. The corresponding topology is called the {\it tree-topology}. This topology is stronger than the usual Euclidean topology. That is, convergence in the tree-topology means also convergence in the Euclidean sense, but the converse is not true. Being an open, closed, or Borel set in the Euclidean metric implies being, respectively, open, closed, or Borel in the tree metric, but not vice versa. When we refer to Borel subsets of $\BR^d$ below, we mean ``Borel in the Euclidean topology". 
The property of boundedness is equivalent in these two metrics; but the property of compactness is not, as described in the following remark. 

\begin{rmk} In the Euclidean topology in $\BR^d$, a closed bounded set is compact. In the tree-topology, this is no longer true in the general case. Here is a counterexample. $\{(1, \theta)\mid \theta \subseteq \Theta\}$ for an infinite subset $\Theta \subseteq \mathbb S$ is closed and bounded, but not compact in the tree topology. However, if a set $A \subseteq \BR^d$ is bounded and closed in the tree topology, and the set $\{\theta \in \mathbb S\mid \exists\, r > 0: r\theta \in A\}$ is finite, then it can be shown that the set $A$ is compact in the tree topology. 
\label{rmk:compact-tree}
\end{rmk}

%Fix a horizon $T > 0$. 
We can define two concepts and spaces of continuity of function $x : [0, T] \to \BR^d$ for every $T > 0$. %, which we use later. 

\smallskip

\noindent (a) Continuity in the Euclidean norm $\norm{\cdot}$; this space is denoted by $C([0, T], \BR^d)$, with the norm
\begin{equation}
\label{eq:max-norm}
\norm{x}_{T} := \max\limits_{t \in [0, T]}\norm{x(t)}, \quad x \in C([0, T], \mathbb R^{d}). 
\end{equation}
For $d = 1$ we simply write $C[0, T]$, instead of $C([0, T], \BR^d)$. 

\smallskip

\noindent (b) Continuity in the tree-metric (\ref{eq:tree}); this space is denoted by $C_{t}([0, T], \BR^d)$, with the metric
$$
\dist_{T}(x, y) := \max\limits_{t \in [0, T]}\dist(x(t), y(t)); \quad x , y \in C_{t}([0, T], \mathbb R^{d}) . 
$$
%\footnote{Picard?}

\smallskip

Fix a Borel subset $B \subseteq \BR^d$; then $\Int B$ denotes the interior of $B$. For any Borel (signed) measure $\nu$ on $B$ and any function $f : B \to \mathbb R$, we denote by $(\nu, f)$ the integral of $f$ over $B$ with respect to $\nu$. Given a Borel measurable function $V : B \to [1, \infty)$,  a finite,  signed Borel measure $\nu$ on $B$ has the following {\it $V$-norm}: 
$$
\norm{\nu}_{V} := \sup\limits_{\substack{f : B \to \BR\\ |f| \le V}}\left|(\nu, f)\right|.
$$
When $V \equiv 1$, this norm is called the {\it total variation norm} and is denoted by 
$\norm{\cdot}_{\TV}$.  

For a continuous function $f : [0, T] \to \BR$, and $\de > 0$, we define the {\it modulus of continuity}:
$$
\omega(f, \de, [0, T]) := \max\limits_{\substack{s, t \in [0, T]\\ |s - t| \le \de}}|f(t) - f(s)|.
$$

For every $f : \BR^d \to \BR$, the radial derivative $f'(r, \theta) $ at $(r, \theta)$ along the ray $\mathcal R_{\theta}$ is defined by 
$$
f'(r, \theta) := \lim\limits_{\eps \downarrow 0}\frac1{\eps}\left[f(r + \eps, \theta) - f(r, \theta)\right],\ \ r > 0, \theta \in \mathbb S. 
$$
For $r = 0$, we can also define such (one-sided) derivative in the direction of $\theta$ at the origin:
\begin{equation}
\label{eq:derivative-at-origin}
f'(0, \theta) \equiv f'(0+, \theta) := \lim\limits_{\eps \to 0}\frac1{\eps}\left[f(\eps, \theta) - f(0, \theta)\right].
\end{equation}
Similarly, we can define $f''$, the second-order radial derivative. 
For every Borel subset $A \subseteq \mathbb S$, we define the function $\chi_A : \BR^d \to \BR$ as follows: 
\begin{equation}
\label{eq:f-A}
%\mathbf{1}_A(r, \theta) := 
\chi_A(r, \theta) := 
\begin{cases}
r,\ \mbox{if}\ r > 0,\ \mbox{and}\ \theta \in A,\\
{0},\ \mbox{otherwise}
\end{cases}
= \ \ 1_A(\theta)r.
\end{equation}

\smallskip

Throughout this article, we operate on a filtered probability space 
$(\Oa, \CF, (\CF_t)_{t \ge 0}, \mathbb P)$, with the filtration satisfying the usual conditions.
The arrow $\Ra$ stands for weak convergence of probability measures or random variables. 
For example, we write $X_{n} \Ra X_{0}$ as $n \to \infty$ for random variables $X_{n}$, $n = 0, 1, 2, \ldots $
The symbol $\mes$ stands for the Lebesgue measure on the real line. 

\smallskip

For an $\BR_+$-valued continuous semimartingale $Y = (Y(t),\, t \ge 0)$ its {\it local time at zero} is  
$$
\Lambda^Y := (\Lambda^Y(t),\, t \ge 0),\quad  \Lambda^Y(t) := \lim_{\varepsilon \downarrow 0} \frac{\,1\,}{\,2 \varepsilon\,} \int^{t}_{0} {1}_{[0, \varepsilon)} ( Y(s)) {\mathrm d} \langle Y \rangle (s)  \, , \quad t \ge 0 . 
$$

\subsection{Definitions of Walsh semimartingales and Walsh diffusions} 
Now, let us take a real-valued continuous semimartingale $U = (U(t),\, t \ge 0)$ with $\, \mes \{ t \ge 0 \mid U(t) \, =\,  0 \} \, =\,  0\,$ a.s. 

\begin{defn} An adapted, continuous (in the tree-topology), $\BR^d$-valued process $X = (X(t),\, t \ge 0)$ is called a {\it  semimartingale on rays driven by} $U$, if 
$$
\mes\{t \ge 0\mid X(t) = \mathbf{0}\} = 0\ \mbox{a.s.},
$$
and the norm ${ \lVert X(\cdot) \rVert}$ of ${X(\cdot)}$ has the following Skorohod decomposition: 
$$
\norm{X(t)} = U(t) + \Lambda^{\norm{X}}(t),\ \ \mbox{where}\ \ \Lambda^{\norm{X}}(t) = \max\limits_{0 \le s \le t}(U(s))_-.
$$
%\footnote{There is no need to define $a_-$, this is standard probability convention}
This process $\Lambda^{\norm{X}} = (\Lambda^{\norm{X}}(t),\, t \ge 0)$, which is the semimartingale local time of $\norm{X}$ at zero, will be also called the {\it local time of} $X$ {\it accumulated  at the origin}. 
\label{defn:smgle-rays}
\end{defn}

Assume $X$ is a semimartingale on rays driven by $U$. For every Borel subset $A \subseteq \mathbb S$, we can consider the following real-valued process: $%\mathbf{1}
\chi_A(X) = (\chi%\mathbf{1}
_A(X(t)),\ t \ge 0)$, where %$\mathbf{1}_A$ 
$\,\chi_{A}\,$ is defined in~\eqref{eq:f-A}. It follows from  \cite[Theorem 2.12(ii)]{MinghanNew} that $\chi_{A}(X) $ 
%$\mathbf{1}_A(X)$ 
is a real-valued, continuous semimartingale. In \cite[Definition 2.13]{MinghanNew} a Walsh semimartingale is defined as follows. 

\begin{defn}[Walsh semimartingale] Take a semimartingale $X$ on rays. Assume in the sense of Definition~\ref{defn:smgle-rays}, there exists a Borel probability measure $\mu$ on $\mathbb S$ such that, for every Borel $A \subseteq \mathbb S$, the local time of $ \Lambda^{ \chi_{A}(X) %\mathbf{1}_A(X)
}$ at the origin satisfies the ``partition of local time'' property a.s. 
\begin{equation}
\label{eq:rays-mu}
\Lambda^{%\mathbf{1}
\chi_A(X)}(t) \equiv \mu(A)\Lambda^{\norm{X}}(t),\ \ t \ge 0.
\end{equation}
Then the semimartingale on rays $X$ is called a {\it Walsh semimartingale} with {\it spinning measure} $\mu$. 
See \cite[Theorem 2.1]{IchibaNew} for a construction of such Walsh semimartingale with (\ref{eq:rays-mu}). 
\label{defn:Walsh-smgle}
\end{defn}

For example, as in Introduction, if $\mathbf{N}$ and $\mathbf{S}$ are North and South Poles, respectively, then a Walsh semimartingale with $\mu = p\delta_{\mathbf{N}} + (1-p)\delta_{\mathbf{S}}$ corresponds to the skew Brownian motion on the real line for $p \in [0, 1]$, where $\,\delta_{\cdot}\,$ is a Dirac measure; and the case $p = 1/2$ corresponds to the usual Brownian motion; and the case $\,p \, =\, 0\,$ or $\,1\,$  corresponds to a reflected Brownian motion. 
%\begin{rmk}
%In some existing literature, for example \cite{MinghanNew}, the term {\it angular measure} is used instead of {\it spinning measure}.  \footnote{is it important as Remark 2?}
%\end{rmk}

Now, let us fix a measurable function $\ell : \mathbb S \to (0, \infty]$ with $\inf_{\mathbb S}\ell > 0$, and define the set 
\begin{equation}
\label{eq:I}
\mathcal I := \{r\theta\mid 0 < r < \ell(\theta),\ \theta \in \mathbb S\}\cup\{\mathbf{0}\}.
\end{equation}
This includes the case when some or all of the values of $\ell(\theta)$ are infinite. For example, if $\ell(\cdot) \equiv + \infty$, then $\mathcal I = \BR^d$. The set $\mathcal I$ from~\eqref{eq:I} is open in the tree-topology, with the boundary
\begin{equation}
\label{eq:boundary-of-I}
\pa\mathcal I := \{\ell(\theta)\theta\mid \theta \in \mathbb S,\ \ell(\theta) < \infty\}.
%\footnote{If $\ell(\theta) = \infty$ then the ray $R_{\theta}$ does not contribute to boundary.}
\end{equation}
Take Borel measurable functions $g : \mathcal I \to \BR$ and $\si : \mathcal I \to (0, \infty)$, and a Borel probability measure $\mu$ on $\mathbb S$. Let $W = (W(t),\, t \ge 0)$ be an $(\CF_t)_{t \ge 0}$-Brownian motion in one dimension. 

\begin{defn} An $\mathcal I$-valued continuous adapted process $X$ is called a {\it Walsh diffusion associated with the triple} $(g, \si, \mu)$, if this is a Walsh semimartingale with spinning measure $\mu$, driven by 
\begin{equation}
\label{eq:defn-Walsh-diffusion}
U(t) := \norm{X(0)} + \int_0^t\left[g(X(s))\, \md s + \si(X(s))\, \md W(s)\right],\ \ t \ge 0.
\end{equation}
In this case, we say that $g$ is the {\it drift coefficient}, $\si^2$ is the {\it diffusion coefficient}, and $\mu$ is the {\it spinning measure} for $X$. For the case $g \equiv 0$ and $\si \equiv 1$, this is called {\it Walsh Brownian motion}, associated with $(0, 1, \mu)$, or simply {\it Walsh Brownian motion with spinning measure} $\mu$. 
%\footnote{definition of $g$ and $\si$ at the origin?}
\label{defn:Walsh-diffusion}
\end{defn}

\begin{rmk} The {\it effective state space} of a Walsh diffusion in $\mathcal I$ associated with $(g, \si, \mu)$, or any Walsh semimartingale with spinning measure $\mu$, is the following set:
\begin{equation}
\label{eq:effective-state-space}
\mathcal I^{\mu} := \{(r, \theta)\mid 0 < r < \ell(\theta),\, \theta \in \supp\mu\}\cup\{\mathbf{0}\}.
\end{equation}
In other words, we need to consider only rays $\mathcal R_{\theta}$ which correspond to  $\theta$ in the support of measure $\mu$. We always start this Walsh diffusion from $X(0) = x \in \mathcal I^{\mu}$, and define the function $\ell$ only on $\supp\mu$, and the coefficients $g$ and $\si$ only on the set~\eqref{eq:effective-state-space}. This distinction becomes  important in section 5 of the current paper. In sections 2-4, we just assume that $\mathcal I$ is the state space.% of this Walsh diffusion.
\label{rmk:effective-state-space}
\end{rmk}

\subsection{Existence and uniqueness of Walsh diffusions} 
Definitions~\ref{defn:Walsh-smgle} and~\ref{defn:Walsh-diffusion} are adapted from \cite{MinghanNew}, and extended to $\BR^d$.  We shall impose some assumptions.

\begin{defn} A function $\phi : \mathcal I \to \BR$ is called {\it locally bounded} if $\sup_{K}|\phi| < \infty$ for every measurable function $R : \mathbb S \to (0, \infty)$ such that 
$$
K := \{r\theta\mid 0 < r \le R(\theta)\}\cup\{\mathbf{0}\} \subseteq \mathcal I,
$$
or, equivalently, $R(\theta) < \ell(\theta)$ for $\theta \in \mathbb S$. 
\label{defn:locally-bounded}
\end{defn}

\begin{asmp} \label{asmp:1} The functions $g, \si, \si^{-1}$ are locally bounded on $\mathcal I$. 
\end{asmp}

Under Assumption 2.1, it was proved in \cite[Section 3]{MinghanNew} that there exists a weak version of the Walsh diffusion on $\mathcal I$, associated with the triple $(g, \si, \mu)$, up to the {\it explosion time}, i.e., the first passage time of $\pa\mathcal I$ in~\eqref{eq:boundary-of-I}. To simplify exposition, we make the following assumption. %We make a stronger assumption: that this hitting moment never happens.

\begin{asmp} \label{asmp:2}
For every initial condition $X(0) = x \in \mathcal I$, there exists a weak version, unique in law, of the Walsh diffusion $X$ in $\mathcal I$, associated with $(g, \si, \mu)$. That is, the moment $\tau$ of hitting the boundary $\pa\mathcal I$ is a.s. infinite, i.e., $\mathbb P(\tau = \infty) = 1$. 
%\footnote{Are you sure you want to merge these two assumptions? It would probably interrupt the natural flow of logic as here. Also, it would require changes all over the paper.}
\end{asmp}

As mentioned earlier, a Walsh diffusion associated with $(g, \si, \mu)$ behaves on any ray $\mathcal R_{\theta}$ as a solution of a one-dimensional SDE with drift $g(\cdot, \theta)$ and diffusion $\si^2(\cdot, \theta)$, as long as it does not hit the origin. When this process hits the origin, it instantaneously chooses the new ray according to the spinning measure $\mu$ (or, more precisely, according to the formula~\eqref{eq:rays-mu}, which is apportioning the local time at the origin between the rays), independently of the past behavior. 
%\footnote{You want some additional citation here?}

\smallskip

It was shown in \cite[Proposition 4.2]{MinghanNew} that under Assumptions \ref{asmp:1} and \ref{asmp:2}, this Walsh diffusion $X = (X(t),\, t \ge 0)$ is, indeed,  a Markov process. 
Denote its transition kernel by $P^t(x, \cdot)$, $\,t \ge 0\,$, $\, x \in \mathbb R^{d}\,$.  Define a family $\mathcal D_{\mathcal I}$ of measurable functions $f : \mathcal I \to \BR$ which satisfy  
the following (a)-(c): %properties:

\smallskip

(a) For every $\theta \in \mathbb S$, the function $f(\cdot, \theta)$ is $C^2((0, \ell(\theta))$, and continuous at $r = 0$;

\smallskip

(b) For every $\theta \in \mathbb S$, the derivative $f'(0+, \theta)$ from~\eqref{eq:derivative-at-origin} is well-defined. Moreover, the function $\theta \mapsto f'(0+, \theta)$ is measurable and bounded on $\mathbb S$;

\smallskip

(c) There exists an $\eps > 0$ such that 
$\, 
\sup_%\limits_{\substack
{\theta \in \mathbb S\\ r \le \eps}%}
|f''(r, \theta)| < \infty.$

\smallskip

This is a more restrictive class of functions than the one described in \cite[Definition 2.6]{MinghanNew}; however, it will suffice for our purposes. 

\medskip

From \cite[Theorem 2.12]{MinghanNew} (suitably adapted to the context of the sphere $\mathbb S$  instead of a unit circle), and Definition~\ref{defn:Walsh-diffusion}, we get the following version of the It\^o-Tanaka formula 
\begin{equation} \label{eq:IT}
\begin{split}
f(X(t)) =&\, f(X(0)) + \int_0^t\left[g(X(s))f'(X(s)) + \frac{1}{2}\si^{2}(X(s))f''(X(s))\right]\md s \\ & + \int_0^t\si(X(s))f'(X(s))\,\md W(s) + \left[\int_{\mathbb S}f'(0+, \theta)\,\mu(\md \theta)\right]\,\Lambda^{\norm{X}}(t) , \quad t \ge 0 
\end{split} 
\end{equation}
for  every function $f \in \mathcal D_{\mathcal I}$.  
Here, the one-dimensional Brownian motion $W$ is taken from~\eqref{eq:defn-Walsh-diffusion}. Using this version of the It\^o-Tanaka formula, one can prove the following statement. %\footnote{Maybe say a bit more?}

\begin{prop}
Under Assumptions 2.1 and 2.2, the generator $\CL$ of the Walsh diffusion associated with $(g, \si, \mu)$ is given by
\begin{equation}
\label{eq:explicit-generator}
\CL f(r, \theta) = g(r, \theta)f'(r, \theta) + \frac12\si^2(r, \theta)f''(r, \theta) , \quad x = (r, \theta) \in \mathbb R^{d} 
\end{equation}
for the following class of functions $f$ in 
\begin{equation}
\label{eq:class-of-functions}
\mathcal D_{\mathcal I, \mu} := \left\{f \in \mathcal D_{\mathcal I}\,\,\biggl\vert\,\,\int_{\mathbb S^d}f'(0+, \theta)\,\mu(\md \theta) = 0\right\}.
\end{equation}
\end{prop}

\begin{rmk} We can also consider the generator $\mathcal L$  in \eqref{eq:explicit-generator} with respect to Euclidean topology; then we need to take functions $f \in \mathcal D_{\mathcal I, \mu}$ on $\BR^d$ which are continuous in Euclidean topology.
\label{rmk:generator-domain-Euclidean}
\end{rmk}

%\begin{proof}
%Apply a function $f \in \mathcal D_{\mathcal I}$ to the Walsh diffusion $X$. 
%For a function $f$ from~\eqref{eq:class-of-functions}, the fourth term in the right-hand side vanishes. The third term is a continuous local martingale. The rest of the proof is similar to the proofs of similar statements about the generator for an SDE on the real line, see for example \cite[Section 7.5]{Oksendal} or any other standard reference. 
%\end{proof}

\subsection{Digression into one-dimensional theory} \label{sec:2.4} 
%\footnote{You want to change the titlde of this subsection? I split this into two subsections.} 
The content of this section is taken from \cite[Section 5.5]{KaratzasShreve}. Consider the one-dimensional stochastic differential equation (SDE)
\begin{equation} 
\label{eq:SDE-1D}
\md Z(t) \, =\,   g(Z(t))\,\md t + \si(Z(t))\,\md W(t),\ t \ge 0,
\end{equation}
where the coefficients $g : \BR \to \BR$ and $\si : \BR \to (0, \infty)$ are locally bounded. 
From Engelbert-Schmidt theory, this type of SDE has a unique in law weak solution up to the explosion time, for every initial condition $Z(0) = z$. One powerful tool to study this type of SDE is the {\it scale function} 
$$
s(x) := \int_0^x\exp\left(-2\int_0^u\frac{g(z)}{\si^2(z)}\,\md z\right)\md u \, , \quad 
$$
for $\, x \in \mathbb R\,$. %and $\, y \in (p(-\infty), p(\infty)) \,$. 
%We assume that the integral inside the exponent is well-defined, that is, $g$ and $\si^2$ are regular enough. For example, it suffices to request that $g$, $\si$, $\si^{-1}$ are bounded on every interval $[-C, C] \subseteq \BR$. 
This scale function is strictly increasing on the whole real line, and so we can define its {\it inverse} $s^{-1} : (s(-\infty), s(\infty))  \to \BR$. %This function is also strictly increasing, but it might not be defined on the whole real line. However, it assumes any real value. 
If we apply the scale function to a solution $Z$ of~\eqref{eq:SDE-1D}, 
we remove the drift coefficient from~\eqref{eq:SDE-1D} and get a continuous local martingale $\tilde{Z} := s(Z) = (s(Z(t)),\, t \ge 0)$. More precisely, the process $\tilde{Z} = s(Z)$ is a solution of the following SDE 
%-
%\newpage
%-
\begin{equation}
\label{eq:scaling}
\md\tilde{Z}(t) = v(\tilde{Z}(t))\,\md W(t) , \quad t \ge 0 , 
\end{equation}
where we define the {\it speed function} 
%\begin{equation}
%\label{eq:speed}
$v(x) := \si(s^{-1}(x))s'(s^{-1}(x))$, $\, x \in \BR$. 
%\end{equation}
%One can say that applying the scale function to a solution of~\eqref{eq:SDE-1D} is a way to remove the drift coefficient from~\eqref{eq:SDE-1D}. 

Next, recall the concept of a time-change to make a Brownian motion from this local martingale $\tilde{Z}$. For simplicity of notation, let us assume that  $g \equiv 0$, and hence $s(x) \equiv x$, and $v(x) \equiv \si(x)$, $x \in \mathbb R $. The time-change is defined as $ \mathcal T(t) := \int_0^t\si^2(Z(s))\,\md s , \quad t \ge 0 $. 
By definition of $\sigma$ this is a strictly increasing function, and one can find a one-dimensional Brownian motion $B = (B(t),\, t \ge 0)$ such that $Z(t) = B(\mathcal T(t))$. Thus, we make a linear Brownian motion from a solution of the one-dimensional SDE~\eqref{eq:SDE-1D} in two steps: (a) removal of drift coefficient by applying the scale function; (b) standardization of diffusion coefficient by applying the time-change.

\subsection{Scale functions and time-change for Walsh diffusions} Same techniques as described in section \ref{sec:2.4} can be used for Walsh diffusions, in principle. However, we need to adjust for dependency of drift and diffusion coefficients on the angular cooordinate $\theta \in \mathbb S$.

First, let us recall the theory of scale functions for Walsh diffusions, developed in \cite[section 3.3]{MinghanNew}. Take a Walsh diffusion on $\mathcal I$ associated with $(g, \si, \mu)$, which satisfies Assumptions 2.1 and 2.2. For $(r, \theta) \in \mathcal I$, define the {\it scale function:}
\begin{equation}
\label{eq:scale}
s(r, \theta) := \int_0^r\exp\left(-2\int_0^{u}\frac{g(z, \theta)}{\si^2(z, \theta)}\,\md z\right)\,\md u.
\end{equation}
Under Assumptions 2.1 and 2.2, the expression~\eqref{eq:scale} is well defined. Moreover, $s(\cdot, \theta)$ is strictly increasing for every $\theta \in \mathbb S$. Thus for every $\theta \in \mathbb S$ there exists an inverse function $s^{-1}(\cdot, \theta)$ such that 
\begin{equation}
\label{eq:inverse-scale}
s(s^{-1}(r, \theta), \theta) \equiv r,\ r \ge 0, \theta \in \mathbb S. 
\end{equation}
Then the function 
\begin{equation}
\label{eq:scale-mapping}
\mathcal P : \mathcal I \ni (r, \theta)   \mapsto (s(r, \theta), \theta) \in \tilde{\mathcal I} := \{(r, \theta)\mid 0 < r < s(\ell(\theta), \theta),\, \theta \in \mathbb S\}\cup\{\mathbf{0}\} 
\end{equation}
is a one-to-one mapping. 
%\begin{equation}
%\label{eq:tilde-I}
%\mathcal I \to \tilde{\mathcal I} := \{(r, \theta)\mid 0 < r < s(\ell(\theta), \theta),\, \theta \in \mathbb S\}\cup\{\mathbf{0}\}.
%\end{equation}
The function ~\eqref{eq:scale-mapping} maps the Walsh diffusion on $\mathcal I$ associated with $(g, \si, \mu)$ into the Walsh diffusion associated with $(0, \tilde{\si}, \mu)$, where the new %diffusion 
coefficient $\,\tilde{\si}\,$ is given by 
\begin{equation}
\label{eq:new-sigma}
\tilde{\si}(r, \theta) = s'(s^{-1}(r, \theta), \theta)\si(s^{-1}(r, \theta), \theta)\ \mbox{for}\ (r, \theta) \in \tilde{\mathcal I}  
\end{equation} 
analogous to the speed function. In other words, just like for the one-dimensional SDE in ~\eqref{eq:SDE-1D}, applying the scale function ~\eqref{eq:scale} to  the drifted  Walsh diffusion would remove the drift coefficient. 

Next, let us make a time-change, as in \cite[section 3.2]{MinghanNew}. Assume for notational convenience that the Walsh diffusion already had zero drift coefficient, that is, it is associated with the triple $(0, \si, \mu)$. Then $s(r, \theta) \equiv r$, and $\tilde{\si}(r, \theta) \equiv \si(r, \theta)$ for $(r, \theta) \in \mathcal I$. Define the time-change
$$
\mathcal  T(t) = \int_0^t\si^2(X(s))\,\md s, \quad t \ge 0.
$$
This is a strictly increasing function, and there exists a Walsh Brownian motion $B = (B(t),\, t \ge 0)$ with spinning measure $\mu$ such that $X(t) \equiv B( \mathcal  T(t))$, $t \ge 0$. 

%Moreover, as mentioned in the survey \cite{Lejay}, convergence of scale functions and speed measures implies weak convergence of these diffusions. A similar statement is true for Walsh diffusions. We shall discuss it later. 

\section{Dirichlet Forms Approach and Stationary Measures} \label{sec:DirichletC}

Another way to define a Walsh diffusion is using Dirichlet forms, via {\it one-point reflection}. This method was designed in \cite[Section 4]{ChenFukushima} to construct Walsh Brownian motion in \cite[Section 5]{ChenFukushima}. It is also developed in \cite[Section 7.5]{ChenBook} and used in \cite[Section 7.6, Example 3]{ChenBook} to construct general Walsh diffusions with finitely supported spinning measure $\mu$. With minor changes, it is applicable to general Walsh diffusions. We shall merely outline the construction here, referring the reader to \cite[Section 7.6, Example 3]{ChenBook} for all details.  A benefit of this method is that it gives us a stationary distribution. Note that this method allows us to define Walsh diffusions starting from $\mes\otimes\mu$-a.e. point $x \in \mathcal I$. Assume we have the same parameters $(\mu, g, \si)$ and the domain $\mathcal I$, as before.  %This is in contrast to the Walsh semimartingale approach, which allows us to define this process starting from any point $x \in \mathcal I$.

\subsection{Construction of a Walsh diffusion using Dirichlet forms} For $\mu$-a.a. $\theta \in \mathbb S$, define the process $\tilde{X}_{\theta}$ on $[0, \ell(\theta))$ which behaves as a solution of an SDE with drift coefficient $g(\cdot, \theta)$ and diffusion coefficient $\si^2(\cdot, \theta)$, absorbed at $x = 0$. 

\begin{asmp} For $\mu$-a.a. $\theta \in \mathbb S$, the process $\tilde{X}_{\theta}$ is conservative.
\end{asmp}

%\begin{rmk} 
Under Assumption 2.1, this Assumption 3.1 is equivalent to the assumption that $\tilde{X}_{\theta}$ does not reach $\ell(\theta)$ in a finite time a.s. 
%\end{rmk}
Under Assumption 3.1 define a process $\tilde{X}$ on $\mathcal I$, as follows: if $\tilde{X}(0) \ne \mathbf{0}$, and if $\theta := \arg(\tilde{X}(0))$, then $\tilde{X}(t) = \tilde{X}_{\theta}(t)\theta$ for $ 0 \le t < \inf \{s: \tilde{X}_{\theta}(s) \, =\, \mathbf{0} \}$. In words: the process $\tilde{X}$ stays on the same ray $\mathcal R_{\theta}$, and evolves there as a one-dimensional diffusion process with drift coefficient $g(\cdot, \theta)$ and diffusion coefficient $\si^2(\cdot, \theta)$, killed at the origin. For notational convenience, we identify the origin $\mathbf{0}$ with the cemetery state $\Delta$, as long as we talk about $\tilde{X}$. 

\smallskip

Fix a $\theta \in \mathbb S$. Let $\tilde{P}_{\theta} = (\tilde{P}^t_{\theta})_{t \ge 0}$ be the transition semigroup of the process $\tilde{X}_{\theta}$. From the theory of one-dimensional SDE, it is known that this semigroup is symmetric with respect to  the measure $\pi_{\theta}$ on $[0, \ell(\theta))$, given by 
\begin{equation}
\label{eq:pi-ray}
\pi_{\theta}(\md r) = \si^{-2}(r, \theta)\exp\left(2\int_0^r\frac{g(\rho, \theta)}{\si^2(\rho, \theta)}\,\md\rho\right)\,\md r.
\end{equation}
This means that the semigroup $\tilde{P}_{\theta}$ satisfies
\begin{equation}
\label{eq:symmetry}
\int_0^{\ell(\theta)}\tilde{P}_{\theta}^tf(r)\, g(r)\, \pi_{\theta}(\md r) = \int_0^{\ell(\theta)}\tilde{P}_{\theta}^tg(r)\, f(r)\, \pi_{\theta}(\md r)
\end{equation}
for all bounded measurable functions $f, g : [0, \ell(\theta)) \to \BR$ with compact support. This additional condition on $f$ and $g$ is introduced to make the integrals in~\eqref{eq:symmetry} finite. Indeed, the measure $\pi_{\theta}$ might be infinite, as for the case of Walsh Brownian motion, $g \equiv 0$ and $\si \equiv 1$.

\smallskip

Let $\tilde{P} = (\tilde{P}^t)_{t \ge 0}$ be the transition semigroup of the process $\tilde{X}$. Define $f_{\theta}(r) := f(r, \theta)$. Then 
\begin{equation}
\label{eq:restriction}
\tilde{P}^tf(r, \theta) = \tilde{P}^t_{\theta}f_{\theta}(r).
\end{equation}
Indeed, as mentioned above, the process $\tilde{X}$ always stays on the same ray $\mathcal R_{\theta}$ and behaves there as the process $\tilde{X}_{\theta}$. Integrating~\eqref{eq:symmetry} with respect to the measure $\mu$ and using~\eqref{eq:restriction}, we get: for any bounded measurable functions $f, g : \mathcal I \to \BR$ with bounded support, 
\begin{equation}
\label{eq:symmetry-new}
\begin{split}
\int_{\mathcal I}\tilde{P}^tf(x)\, g(x)\, \pi(\md x) &= \int_{\mathbb S} \int^{\ell(\theta)}_{0} \tilde{P}_{\theta}^{t} f_{\theta}(r) g(r, \theta) \pi_{\theta} ({\mathrm d} r) \mu ({\mathrm d} \theta)  \\
& =\int_{\mathbb S} \int^{\ell(\theta)}_{0} \tilde{P}_{\theta}^{t} g_{\theta}(r) f(r, \theta) \pi_{\theta} ({\mathrm d} r) \mu ({\mathrm d} \theta) = \int_{\mathcal I}\tilde{P}^tg(x)\, f(x)\, \pi(\md x),
\end{split}
\end{equation}
where the new measure $\pi$ on $\mathcal I$ is defined as
\begin{equation}
\label{eq:stationary-measure}
\pi(\md r, \md \theta) = \pi_{\theta}(\md r)\,\mu(\md \theta) = \si^{-2}(r, \theta)\exp\left(2\int_0^r\frac{g(\rho, \theta)}{\si^2(\rho, \theta)}\,\md\rho\right)\,\md r\,\mu(\md\theta),\ \pi(\{ \bm{0}\}) = 0.
\end{equation}
Therefore, $\tilde{X}$ is symmetric with respect to $\pi$. Let us translate the notation from \cite{ChenFukushima} into our notation. The state space is $E := \mathcal I$. The lifetime $\zeta$ of $\tilde{X}$ is the first hitting time of the origin: 
$$
\zeta := \inf\{t \ge 0\mid \tilde{X}(t) = \mathbf{0}\}.
$$
Indeed, the process $\tilde{X}$ can get killed only by reaching the origin. The reason for this is that for every $\theta \in \mathbb S$, the process $\tilde{X}_{\theta}$ is conservative, that is, it does not reach $\ell(\theta)$ in finite time a.s. The new point $a$ is the origin: $a = \mathbf{0}$. We denote Markov transition probabilities for the process $\tilde{X}$ by $\tilde{P}_x(A) = \mathbb P(A\mid \tilde{X}(0) = x)$ for $x \in \mathcal I$ and Borel subsets $A \subseteq \mathcal I$.   Applying the one-point reflection construction from \cite[section 7.5]{ChenBook}, we get a Walsh diffusion associated with the triple $(g, \si, \mu)$, as a Markov process with an infinitesimal generator~\eqref{eq:explicit-generator}. It is conservative (i.e., non-exploding), and its transition semigroup $(P^t)_{t \ge 0}$ satisfies $P^t1 = 1$, $t \ge 0$. 

\smallskip

Under Assumption 2.1, let us now present Assumption 2.2 in a slightly weakened form.

\begin{asmp} For $\mes\otimes\mu$-a.e. starting point $x \in \mathcal I$, the Walsh diffusion on $\mathcal I$ associated with $(g, \si, \mu)$ is conservative. 
\end{asmp}

\begin{lemma} Under Assumption 2.1, Assumption 3.1 is equivalent to Assumption 3.2.
\end{lemma}

\begin{proof} Under Assumptions 2.1 and 3.1, non-explosiveness of the Walsh diffusion follows from the construction above. Conversely, suppose Assumption 3.1 were not satisfied for some $\theta \in A$, where $A \subseteq \mathbb S$ is a Borel subset with $\mu(A) > 0$. Start the Walsh diffusion associated with $(g, \si, \mu)$ from any point $x = (r, \theta) \in (0, \infty)\times A$. Then with positive probability it does not exit the ray $\mathcal R_{\theta}$, and hence with positive probability it would reach the point $\ell(\theta)\theta \in \pa\mathcal I$. Thus, this Walsh diffusion would not be conservative for this starting point, although the set $(0, \infty)\times A$ of such points has positive $\mes\otimes\mu$-measure, which contradicts Assumption 3.2. Thus, these assumptions are equivalent. 
%Assumption 3.1 is equivalent to Assumption 3.2. 
\end{proof}

\subsection{Stationary distributions and measures} Take an $\BR^d$-valued continuous-time Markov process  $X = (X(t),\, t \ge 0),$ with transition kernel $P^t(x, \cdot)$. We shall distinguish two measures below. 

\begin{defn}[Stationary distribution] \label{defn:stationary}
We say that a Borel probability measure $\pi$ on $\BR^d$ is a {\it stationary distribution} for $X$ if the process $\,X\,$ starting from the initial distribution $\pi$, forever remains at the same distribution $\pi$ (we write $X(t) \sim \pi$), that is,  $X(t) \sim \pi$ for every $t \ge 0$, if $X(0) \sim \pi$; Or equivalently for every bounded measurable function $f : \BR^d \to \BR$, and every $t > 0$
%Or, equivalently, if for every bounded measurable function $f : \BR^d \to \BR$, and every $t > 0$, we have:
\begin{equation}
\label{eq:stationary-def}
(\pi, P^tf) = (\pi, f), 
\end{equation}
where we use the notation $\, (\mu, f) \, =\, \int_{\mathbb R^{d}} f(x) \mu ( {\mathrm d} x) \,$ of $\,f\,$ with respect to the measure $\, \mu\,$. 
\end{defn}

\begin{defn}[Stationary measure] 
A $\si$-finite Borel measure $\pi$ on $\BR^d$ with $\pi(K) < \infty$ for bounded subset $K$ is called a {\it stationary measure} for the Markov process $X$ if the equality~\eqref{eq:stationary-def} holds for every bounded measurable function $f : \BR^d \to \BR$ with bounded support. 
\end{defn}

%Sometimes the stationary measure is infinite (that is, not a probability distribution), but it will always be finite on bounded subsets. Thus, we need to modify Definition~\ref{defn:stationary} to incorporate infinite measures. 

Now take $X = (X(t),\, t \ge 0)$ to be the Walsh diffusion on $\mathcal I$ associated with $(g, \si, \mu)$. By construction in section \ref{sec:DirichletC}, the new process $X$ is symmetric with respect to the measure $\pi$ from~\eqref{eq:stationary-measure}. That is, if $(P^t)_{t \ge 0}$ is its transition semigroup, then for any bounded measurable function $f, g : \mathcal I \to \BR$ with bounded support, and any $t > 0$, 
%-
%\newpage
%-
\begin{equation}
\label{eq:symmetry-Walsh}
\int_{\mathcal I}(P^tf)(x)\,g(x)\,\pi(\md x) = \int_{\mathcal I}f(x)\, (P^tg)(x)\,\pi(\md x). 
\end{equation}
Since $X$ is conservative and it does not explode, i.e., $P^t1 = 1$, letting $g \equiv 1$ in~\eqref{eq:symmetry-Walsh}, we obtain  

\begin{thm}
\label{thm:Walsh-stat}
The measure $\pi$ from~\eqref{eq:stationary-measure} is a stationary measure for the Walsh diffusion on $\mathcal I$ associated with the triple $(g, \si, \mu)$. 
\end{thm}

The following discussion clarifies a ``physical sense'' of stationary measure $\pi$ with $\pi(\mathcal I) < \infty$. We may normalize $\pi$ to be a stationary distribution. For $\mu$-almost all $\theta \in \mathbb S$, $\pi_{\theta}([0, \ell(\theta))) < \infty$. Now, fix a $\theta \in \mathbb S$. Take a reflected one-dimensional diffusion on $[0, \ell(\theta))$ with drift coefficient $g(\cdot, \theta)$ and diffusion coefficient $\si^2(\cdot, \theta)$, reflected at $0$, which, by Assumption 3.1, never hits $\ell(\theta)$. Then this process has a unique stationary distribution 
\begin{equation}
\label{eq:1D-stationary}
\pi_{\theta}(\md r) = C^{-1}(\theta) \si^{-2}(r, \theta)\exp\left(2\int_0^r\frac{g(\rho, \theta)}{\si^2(\rho, \theta)}\md \rho\right) {\mathrm d} r 
 =: C^{-1}(\theta)p(r, \theta)\md r, \quad (r, \theta) \in \mathcal I, 
\end{equation}
%where we define the density function:
%with 
%\begin{equation}
%\label{eq:density}
%p(r, \theta) = \si^{-2}(r, \theta)\exp\left(2\int_0^r\frac{g(\rho, \theta)}{\si^2(\rho, \theta)}\md \rho\right),
%\end{equation}
%and 
with the normalizing constant 
%\begin{equation}
%\label{eq:normalizing}
$ C(\theta) := \int_0^{\ell(\theta)}p(r, \theta)\,\md r$.
%\end{equation}
The stationary distribution can then be %written as follows:
\begin{equation}
\label{eq:interpret}
\pi(\md r, \md\theta) = C(\theta)\,\pi_{\theta}(\md r)\,\mu(\md\theta), \quad (r, \theta) \in \mathcal I . 
\end{equation}

\begin{rmk} An informal interpretation of~\eqref{eq:interpret} is as follows. The stationary distribution of the given Walsh diffusion is a combination of all radial one-dimensional stationary distributions of the radial processes, weighted by the spinning measure $\mu$ (governing how often at the origin the process chooses a given direction), and by $C(\theta)$, the average excursion time for the ray $\mathcal R_{\theta}$.

\smallskip

To simplify, let us %impose Assumption \ref{asmp:4}: 
assume the support 
$\supp\mu = \{\theta_1, \ldots, \theta_p\}$ is finite. Then~\eqref{eq:interpret} becomes 
\begin{equation}
\label{eq:interpret-finite-support}
\pi(\md r, \{i\}) = C_i\, \mu_i \, \pi_i(\md r),\ \ i = 1, \ldots, p,\ \ r \in [0, \ell_i].
\end{equation}
Here, $C_i := C(\theta_i),\, \pi_i := \pi_{\theta_i},\, \ell_i := \ell(\theta_i)$ and $\mu_{i} := \mu (\theta_{i})$ for $i =1, \ldots , p$. Let us interpret the stationary distribution as a long-term average occupation time. %, as in section 5. 
Then for $A \subseteq [0, R_i]$ %, %combining~\eqref{eq:interpret-finite-support} with~\eqref{eq:running-average-convergence}, we get:
$$
\pi(A\times\{i\}) = \lim\limits_{T \to \infty}\frac 1T\int_0^T1\left(\arg X(s) = \theta_i,\, \norm{X(s)} \in A\right)\,\md s , \quad i = 1, \ldots , p, 
$$
thanks to the average occupation times formula. In particular, letting $A = [0, \ell_i]$, we get 
$$
C_i\mu_i = \pi([0, \ell_i]\times\{i\}) = \lim\limits_{T \to \infty}\frac1T\mes\{s \in [0, T]\mid \arg X(s) = \theta_i\}.
$$
Here, $\mu_i$ is the long-term proportion of the times this Walsh diffusion chooses the ray $\mathcal \ell_i := \mathcal \ell_{\theta_i}$, and $C_i$ is the factor corresponding to the average time spent on this ray $\mathcal \ell_i$, for each such excursion. 
\end{rmk}

\begin{exm} \label{ex:1} Let $\mathcal I := \BR^d$, $g(r, \theta) \equiv g(\theta) < 0$, $\si(r, \theta) \equiv \si(\theta) > 0$ for every $\, r > 0 \,$, $\, \theta \in \mathbb S\,$. Denote $\la(\theta) := -2g(\theta)\si^{-2}(\theta) > 0 $ for $\, \theta \in \mathbb S\,$. Then direct calculation gives us the scale function $\, s( \cdot)\,$ and the stationary measure $\, \pi (\cdot)\,$ 
$$
s(r, \theta) = \frac1{\la(\theta)}\left(e^{\la(\theta)r} - 1\right),\quad 
\pi(\md r, \md\theta) = \si^{-2}(\theta)\exp\left(-\la(\theta)r\right)\,\md r\,\mu(\md \theta) , \quad r > 0, \, \, \theta \in \mathbb S. 
$$
\end{exm}

\section{Continuous Dependence on Spinning Measure and Initial Condition} 

\subsection{Main results} 
%\footnote{Upper and lower estimates are combined, the first subsection has all results, the other three have proofs.} 
In this section, using Euclidean distance $\norm{\cdot}_T$ from~\eqref{eq:max-norm} for $C([0, T], \BR^d)$, we show that the law of the Walsh diffusion on $\mathcal I$ associated with $(g, \si, \mu)$, starting from $X(0) = x$, continuously (in the weak sense) depends on the measure $\mu$, and on the initial condition $x$. %In this section, we use Euclidean distance for the functional space. 
%\footnote{Change to italic?} %We impose an additional assumption on $g$ and $\si$. 

\begin{asmp} \label{asmp:3}
The functions $g, \si : \mathcal I \to \BR$ are continuous in the Euclidean topology.  Moreover, the function $\ell : \mathbb S \to (0, \infty]$ is lower semicontinuous, that is, for every $\theta \in \mathbb S$, 
$$
\varliminf\limits_{\theta \to \theta_0}\ell(\theta) \ge \ell(\theta_0) . 
$$
\end{asmp}

\begin{rmk} The assumption on $\ell$ can be equivalently stated as follows: the function $\ell$ can be pointwise approximated by an increasing sequence $(R_n: \mathbb S \to (0, \infty))_{n \ge 1}$ of continuous functions %$R_n : \mathbb S \to (0, \infty)$: 
$$
R_n(\theta) \uparrow \ell(\theta)\ \mbox{as}\ n \to \infty\ \mbox{for every}\ \theta \in \mathbb S.
$$
In this case, define
\begin{equation}
\label{eq:D-m}
D_m := \{(r, \theta)\mid 0 < r \le R_m(\theta),\, \theta \in \mathbb S\}\cup\{\mathbf{0}\},\ m \ge 1. 
\end{equation}
Then $(D_m)_{m \ge 1}$ is an increasing sequence of compact (in a Euclidean topology) subsets of $\mathcal I$ with each interior
$
\Int D_m :=  \{(r, \theta)\mid 0 < r < R_m(\theta),\, \theta \in \mathbb S\}\cup\{\mathbf{0}\} \,.
$
By construction, this sequence satisfies the properties
$$
\Int D_m \subseteq \Int D_{m+1},\ \ \mbox{and}\ \ \bigcup\limits_{m=1}^{\infty}\Int(D_m) = \mathcal I. 
$$
Assumption~\ref{asmp:3} holds, for example, when $\ell(\cdot) \equiv \infty$, or when $\ell : \mathbb S \to (0, \infty)$ is  continuous. Assumption \ref{asmp:3} does not hold, for example, %Here is an example of the function $\ell$ which violates  %. for %which it does not hold is described as follows: 
$\, \ell(\theta) :=  1 + {\bf 1}_{\{\theta = \theta_{0}\}} \,$ for $d \ge 2$ and a fixed point $\theta_0 \in \mathbb S$. 
%and try 
%$$
%\ell(\theta) = 
%\begin{cases}
%1,\, \theta \ne \theta_0;\\
%2,\, \theta = \theta_0.
%\end{cases}
%$$
\label{rmk:particular-ell} 
\end{rmk}

\begin{rmk} Assumption \ref{asmp:1} actually follows from Assumption \ref{asmp:3}. Indeed, take a function $R$ as in Definition~\ref{defn:locally-bounded}, and approximate the function $\ell$ by $R_{n}$, $n \ge 1$ as in Remark \ref{rmk:particular-ell}. Then we have  
$$
R(\theta) < \ell(\theta) = \lim\limits_{n \to \infty}R_n(\theta),\ \ \forall\ \ \theta \in \mathbb S.
$$
Thus the open sets $O_n := \{\theta \in \mathbb S\mid R(\theta) < R_n(\theta)\}$, $n \ge 1$ form an open cover of $\mathbb S$. By compactness, we can extract a finite subcover $O_{n_1}, \ldots, O_{n_j}$ of $\mathbb S$. Then there exists an $m = \max(n_1, \ldots, n_j)$ such that for all $\theta \in \mathbb S$ we have $R(\theta) < R_m(\theta)$. Therefore, $K \subseteq D_m$, where $K $ is defined in Definition \ref{defn:locally-bounded}, and $D_m$ is taken from~\eqref{eq:D-m}. Since $g$ is continuous and $D_m$ is compact in Euclidean topology, we conclude $g$ is bounded on $D_m$, and so is on $K$. The same argument applies to $\si$ and $\si^{-1}$. 
\end{rmk}

%Our main result is as follows. Fix a set $I$ from~\eqref{eq:I}. 

\begin{thm} Let us consider the  Walsh diffusion $X_n$ on $\mathcal I$ associated with $(g, \si, \mu_n)$, starting from $X_n(0) = x_n$, $n \ge 0$. For every $n = 0, 1, 2, \ldots$, take a point $x_n \in \mathcal I$ in ~\eqref{eq:I} and a spinning measure $\mu_n$ on $\mathbb S$. Suppose that the functions $g$ and $\si$, and the domain $\mathcal I$, satisfy Assumption \ref{asmp:3}. Suppose also that every Walsh diffusion $X_n$ satisfies Assumption \ref{asmp:2}. If $\mu_n \Ra \mu_0$, and $x_n \to x_0$, then $X_n \Ra X_0$ in $C([0, T], \BR^d)$. 
%\footnote{Are you sure you want to write $Q_T(x_n, \mu_n)$? We do not state Wasserstein distance estimate here, and weak convergence can be stated in terms of processes rather than their distributions.}
\label{thm:continuous-dependence-on-spinning-measure}
\end{thm}

%The rest of this section is devoted to the proof of this result. First, we prove it for the case of Walsh Brownian motion. In this case, we can refine this result by estimating the speend of convergence. We also present a lower bound on the speed of convergence. Finally, we prove the result in full generality. 

%\subsection{The case of Walsh Brownian motions: Wasserstein distance estimate}

Let us state separately this convergence result for Walsh Brownian motions in $\BR^d$. For this case, $g \equiv 0$, $\si \equiv 1$, and $\ell \equiv \infty$, and Theorem~\ref{thm:continuous-dependence-on-spinning-measure} takes the following form. 

%We would like to show continuous dependence of the distribution of Walsh Brownian motion on the spinning measure, that is, if the spinning measures $\mu$ and $\ol{\mu}$ are close, then the distributions of the corresponding Walsh Brownian motions are close too. Our statement therefore is as follows.

\begin{cor} \label{cor:Walsh-BM-continuous-dependence}
For every $n = 0, 1, 2, \ldots$, take a Walsh Brownian motion $W_n$ with spinning measure $\mu_n$, starting from $W_n(0) = x_n$. If $x_n \to x_0$, and $\mu_n \Ra \mu_0$, then $W_n \Ra W_0$ in $C([0, T], \BR^d)$.
\end{cor}

%To this end, we use the Wasserstein distance defined in \eqref{eq:Wasserstein-defn} on the metric space $\, (\mathfrak X, {\bm d})\,$. For measures $\mu$ and $\ol{\mu}$ on $\mathbb S$, we use the standard Euclidean distance  in $\BR^d$ as the reference metric in (\ref{eq:Wasserstein-defn}); For measures $\mathbb Q_T(\mu)$ and $\mathbb Q_T(\ol{\mu})$, we use the canonical uniform norm $\norm{x}_{T} := \max_{0 \le t \le T}\norm{x(t)}$, $x \in C([0, T], \mathbb R^{d})$ as the reference metric in (\ref{eq:Wasserstein-defn}). 

%Let $C([0, T], \BR^d)$ be the space of all continuous functions $x : [0, T] \to \BR^d$, with the canonical uniform norm. For $d = 1$, we simply write $C([0, T], \BR^1) =: C([0, T])$. 

%In this subsection, we shall show that the law of the Walsh Brownian motion $X$ continuously depends on the spinning measure $\mu$. More precisely, consider the sequence $(\mu_n)_{n \ge 0}$ of probability measures on $\mathbb S$, and for each $n = 0, 1, 2, \ldots$ let $X_n$ be the corresponding Walsh Brownian motion in $\BR^d$ with spinning measure $\mu_n$ on a filtered probability space. Then we have the following weak convergence result. 
%
%\begin{thm} \label{Thm1} If $\mu_n \Ra \mu_0$ in $\mathbb S$, then $X_n \Ra X_0$ in $C([0, T], \BR^d)$. 
%\end{thm}
Consider the case when all these Walsh Brownian motions in $\BR^d$ start from the origin: $x_n = \mathbf{0}$, $n = 0, 1, 2, \ldots$. Then we can actually quantify the rate of convergence in the following Theorem~\ref{thm:main-convergence}. For two Borel probability measures $\nu_1$ and $\nu_2$ on a metric space $(\mathcal X, {\mathbf d})$ with metric $\mathbf d$, the {\it Wasserstein distance of order} $p \ge 1$ is defined as follows:
\begin{equation}
\label{eq:Wasserstein-defn}
W_p(\nu_1, \nu_2) := \inf_{\gamma \in \Gamma}\left[\, \int_{\mathcal X \times \mathcal X}  \, \big (\, {\mathbf d}(x_1, x_2)  \, \big)^p \, \md \gamma (x_{1}, x_{2})  \, \right]^{1/p},
\end{equation}
where the infimum is taken over the family $\,\Gamma\,$ of probability measures on $ \mathcal X \times \mathcal X$ with marginals $\nu_1$ and $\nu_2$ for which the integral inside the bracket is finite. It is known from \cite{Rachev91, Villani08} that convergence in the Wasserstein distance implies the weak convergence. Thus we shall estimate the Wasserstein distance between the distributions of $X_n$ and $X_0$, as random elements of $C([0, T], \BR^d)$, using the Wasserstein distance between $\mu_n$ and $\mu_0$ on $\mathbb S$. 

Denote by $\mathbb Q_T(\mu)$ the law of the Walsh Brownian motion starting from the origin with spinning measure $\mu$ in the space $C([0, T], \BR^d)$. Theorem \ref{Thm2.1} provides %both lower and upper  
the upper estimate of the Wasserstein distance between $\mathbb Q_T(\mu)$ and $\mathbb Q_T(\ol{\mu})$, for two spinning measures $\mu$ and $\ol{\mu}$ on $\mathbb S$. 

\begin{thm} \label{Thm2.1} Take two Borel probability measures $\mu$ and $\ol{\mu}$ on $\mathbb S$. 
%
%\medskip
%
%\noindent %(a) 
For all positive constants $p, q, \rho, T$ with  
\begin{equation} \label{1pqr}
1 \le q < p,\ \ \rho \in \Big( 0 , \frac{p}{p+1}\Big),
\end{equation}
there exists a positive constant $C^{\ast}$ dependent on $\,T, p, q, \rho\,$, such that %the Wasserstein distances in (\ref{eq: W2}) satisfy 
\begin{equation}
\label{eq:main-ineq}
W_q\left(\mathbb Q_T(\mu), \mathbb Q_T(\ol{\mu})\right) \le C^{\ast}\left[W_p(\mu, \ol{\mu})\right]^{\rho}.
\end{equation}

\medskip

%\noindent (b) For all positive constants $T$ and $q \ge 1$, there exists a positive constant $C_{\ast}$, dependent on $q$ and $T$, such that 
%\begin{equation} \label{eq:LB}
%C_{\ast} \, W_{q} ( \mu, \overline{\mu}) \le W_{q} \big( \mathbb Q_{T}(\mu), \mathbb Q_{T} (\overline{\mu})\big) \, . 
%\end{equation}

\label{thm:main-convergence}
\end{thm}

The rest of this section is organized as follows. 
We shall prove Theorem~\ref{thm:main-convergence} in section 4.2, Corollary~\ref{cor:Walsh-BM-continuous-dependence} in section 4.3 and then complete the proof of Theorem~\ref{thm:continuous-dependence-on-spinning-measure} in section 4.4.  

%%In section 4.2, we shall prove Theorem~\ref{thm:main-convergence}. 
%%: the upper bound~\eqref{eq:main-ineq} first and then the lower bound~\eqref{eq:LB}. 
%In section 4.3, we use~\eqref{eq:main-ineq} to prove Corollary~\ref{cor:Walsh-BM-continuous-dependence}. Finally, in section 4.4, we use Corollary~\ref{cor:Walsh-BM-continuous-dependence} to complete the proof of Theorem~\ref{thm:continuous-dependence-on-spinning-measure}. 
%%The proof is postponed until subsection 3.4.  

%%%%%%%%%%%%%%%%%% added by Tomoyuki on 2016/08/25
%\subsection{Lower bound}
%We shall show a converse of the inequality (\ref{eq:main-ineq}), that is, for every $\, q \ge 1\,$ there exists a positive constant $\,C_{L}\,$ such that 

\subsection{Proof of Theorem~\ref{thm:main-convergence}} 

\noindent {\it Step 1: Efficient coupling.} Let us consider the following Walsh Brownian motions $X = (X(t), t \ge 0)$ and $\ol{X} = (\ol{X}(t), t \ge 0)$ with spinning measures $\mu$ and $\ol{\mu}$, respectively, coupled by the following procedure. 

\smallskip 

On a filtered probability space, let us take a reflected Brownian motion $S = (S(t), t \ge 0)$ with values in $\,[0, \infty)\,$, starting from zero, instantaneously reflected at zero. 
%Informally, for every excursion of $S$, a new random variable $\theta $ distributed in $\mu$ is generated independently of $S$ and of each other. Then 
%we define  $X(t) := \theta \, S(t)$ for each time $t$ in this excursion interval of $S$. We do the same for $\ol{X}$, that is, $\ol{X}(t) := \ol{\theta}\,  S(t)$, where $\ol{\theta}$ is a random variable distributed in $\ol{\mu}$, independent of $S$ and of each other for each $t$ which the excursion straddles. To achieve an efficient coupling of $X$ and $\ol{X}$, in terms of smaller Wasserstein distance, we take an efficient coupling %$(\theta, \ol{\theta})$ 
%of $\mu$ and $\ol{\mu}$. 
%\smallskip 
Let $e = (e(t), t \ge 0)$ be the excursion process, so that $e(t)$ is the excursion of the reflected Brownian motion $S$ at time $t$, and denote by $\mathcal J$ the (countable) set of its distinct elements. This set depends on $\omega \in \Omega$. The notation is taken from \cite[Chapter 12]{RevuzYorBook}. Take a certain coupling $\Theta$ of marginal probability measures $\mu$ and $\ol{\mu}$. Generate a sequence $(\theta_j, \ol{\theta}_j) $ of $\mathbb S \times \mathbb S\,$-valued, independently, identically distributed random variables jointly distributed in $\Theta$, indexed by $j \in \mathcal J$. Define the Walsh Brownian motions $X$ and $\ol{X}$ as follows: for  each $t \ge 0$, if $S(t) = 0$, then $X(t) \, :=\,  0 \, =:\, \ol{X}(t) $, and if $S(t) > 0$, then there exists a unique index $j \in \mathcal J$, such that $e(t) = j$, and we let 
\begin{equation} \label{eq: XX}
X(t) := \theta_j\, S(t),\quad \ \ \ol{X}(t) := \ol{\theta}_j\, S(t) 
\end{equation}
on an extended, filtered probability space. Such construction of Walsh Brownian motions (and semimartingales) %\footnote{Remove this "more generally"} 
is recently examined and described in \cite{IchibaNew}. 
%\footnote{Would you like to change this?}

Note that we use the common reflected Brownian motion $S$ in the construction, and so the resulting Walsh Brownian motions $X$ and $\ol{X}$ have the same time intervals for excursions, i.e., $\, 
\{t : X(t) \, =\, {\bf 0} \} \, =\,  \{t : \overline{X}(t) \, =\,  {\bf 0}\} \,$ a.s.
This procedure creates two Walsh Brownian motions $X$ and $\ol{X}$ with spinning measures $\mu$ and $\ol{\mu}$, respectively. In other words, the probability measure induced by this pair $(X, \ol{X})$ is a coupling $\, \Pi (\Theta)\,$ of marginal probability distributions $\mathbb Q_T(\mu)$ and $ \mathbb Q_T(\ol{\mu})$ in $\,C([0, T], \mathbb R^{d})\,$, where $\, \mu\,$ and $\, \overline{\mu}\,$ are the marginal of the coupling $\, \Theta\,$. To achieve an efficient coupling $\,\Pi (\Theta)\,$ of $X$ and $\ol{X}$, in the sense of smaller Wasserstein distance, we take an efficient coupling $\, \Theta\, $ %$(\theta, \ol{\theta})$ 
of $\mu$ and $\ol{\mu}$. 

Let us denote by $\, \Gamma_{0}\,$ all couplings of marginal probability measures $\, \mu\,$ and $\, \overline{\mu}\,$. Also, let us denote by $\,\Gamma_{1}\,$ the family of probability measures, each of which is induced by the coupling $\, (X, \overline{X})\,$ of distributions $\, \mathbb Q_{T}(\mu)\,$ and $\, \mathbb Q_{T}( \overline{\mu})\,$, constructed in the above procedure from a coupling $\,\Theta\,$ of $\, \mu\,$ and $\, \overline{\mu} \,$ in $\, \Gamma_{0}\,$, i.e., $\,\Gamma_{1} \, :=\, \{\Pi (\Theta) : \Theta \in \Gamma_{0}\}\,$. Then by definition (\ref{eq:Wasserstein-defn}) for $\,p, q \ge 1\,$ here we shall evaluate the Wasserstein distances  
\begin{equation} \label{eq: W2}
\begin{split}
W_p(\mu, \ol{\mu}) \, &:=\,  \big( \inf_{\Theta \in \Gamma_{0}} \mathbb E^{\Theta}[ \,  \lVert \theta - \overline{\theta} \rVert^{p} \, ] \, \big)^{1/p} \, , \\
W_q\left(\mathbb Q_T(\mu), \mathbb Q_T(\ol{\mu})\right) \, &\le\, \big( \inf_{P \in \Gamma_{1}} \mathbb E^{P}[ \, \lVert \, X - \overline{X} \,  \rVert_{T}^{q} \, ] \, \big)^{1/q} \, =\, \big( \inf_{\Theta \in \Gamma_{0}} \mathbb E^{\Pi(\Theta)}[ \, \lVert \, X - \overline{X} \,  \rVert_{T}^{q} \, ] \, \big)^{1/q} \, , 
\end{split} 
\end{equation}
where $\, \mathbb E^{\Theta}\,$ and $\, \mathbb E^{P} \, =\, \mathbb E^{\Pi(\Theta)}\,$ are expectations under $\,\Theta \in \Gamma_{0}\,$, $\, P \, =\,  \Pi(\Theta) \in \Gamma_{1}\,$, respectively. To this end, given the constants $\, p, q, T\,$ in (\ref{1pqr}), we shall 
estimate the upper bound of 
\begin{equation}\label{eqEXX0}
\mathbb E[ \lVert X - \overline{X}\rVert_{T}^{q}] \, =\,  \mathbb E^{\Pi(\Theta)} [ \max\limits_{0 \le t \le T}\norm{X(t) - \ol{X}(t)}^q ] \,, 
\end{equation}
in terms of $ \mathbb E^{\Theta} [ \norm{\theta - \ol{\theta}}^p] $, for the coupling $\, \Pi (\Theta) \,$ of the pair $\, X\,$ and $\, \overline{X}\,$ of Walsh Brownian motions, and for the coupling $\,\Theta \,$ of the spinning measures $\, \mu\,$ and $\, \overline{\mu}\,$, described in (\ref{eq: XX}).  

\medskip 

The idea of the proof is as follows. Take a constant $\de (> 0)$. There are two kinds of excursions of $S$: the first kind consists of excursions with height greater than $\de$, and the second kind consists of excursions with height less than or equal to $\de$. 
There are at most finite number of excursions of the first kind. 
%$j \in \mathcal J$ with height greater than $\de$. 
For them, we have 
\begin{equation} \label{XtS}
\norm{X(t) - \ol{X}(t)} \le \norm{\theta_j - \ol{\theta}_j}\, S(t), \quad t \ge 0 \, , 
\end{equation}
where we can estimate the running maximum of $S$ from above. Take a time moment $t$ corresponding to the second kind of excursions. By  the triangle inequality and the construction in (\ref{eq: XX}), we have 
\begin{equation} \label{X2d}
\norm{X(t) - \ol{X}(t)} \le \norm{X(t)} + \norm{\ol{X}(t)} = S(t) + S(t) \le 2\, \de.
\end{equation}
With this idea of separating excursions in two kinds, we estimate the upper bound of (\ref{eqEXX0}), and then minimize it by choosing $\,\delta\,$. %Now let us carry out the proof in two steps. 

\medskip 

\noindent {\it Step 2: Excursions.} Define the set of excursions of $S$ 
$$
\mathcal J_T := \{j \in \mathcal J\mid \exists t \in [0, T]:\ e(t) = j\} 
$$
restricted to the time interval $[0, T]$,  including the last excursion (which is sometimes called a {\it Brownian meander}). This last excursion can start at a time moment $t' \le T$, but end at $t'' > T$; this excursion is included if $\max_{t \in [t', T]}S(t) \ge \de$. Let us classify it into two kinds of excursions   
\begin{equation}
\mathcal J_{T, \de} := \{j \in \mathcal J_T\mid H(j) > \de\},\quad \mathcal J_{T, \de}^{c} := \mathcal J_{T} \setminus \mathcal J_{T, \delta} \, , 
\label{eq:excursion-height}
\end{equation} 
where $H(j)$ is the height of the excursion $j \in \mathcal J$. Since $\mathcal J_{T, \de}$ is a finite set, by (\ref{X2d})-(\ref{XtS}), 
$$
\max\limits_{0 \le t \le T}\norm{X(t) - \ol{X}(t)}^q  \le \max\limits_{{\{t: e(t) \in \mathcal J_{T, \delta}\}}} \norm{X(t) - \ol{X}(t)}^q + \max\limits_{{\{t: e(t) \in \mathcal J_{T, \delta}^{c}\}}}\norm{X(t) - \ol{X}(t)}^q
$$
$$
\le \big( \max\limits_{0 \le t \le T}S(t) \big)^q  \cdot \max\limits_{j \in \mathcal J_{T, \de}}\norm{\theta_j - \ol{\theta}_j}^q  \, {} + (2\de)^q .
$$

If the set $\mathcal J_{T, \de}$ is empty, we let the maximum of zero numbers to be zero. Taking the expected values and applying H\"older's inequality  to the product in the right hand side with
$$
r_1 := \frac{p}{p-q},\ \ r_2 := \frac pq,\ \ \frac1{r_1} + \frac1{r_2} = 1\,, 
$$
we obtain the first upper bound of (\ref{eqEXX0})
\begin{equation}
\label{eq:X-X}
\mathbb E^{\Pi(\Theta)} \big[ \max\limits_{0 \le t \le T}\norm{X(t) - \ol{X}(t)}^q \big] \le (2\de)^q + C_1 \cdot \left[\mathbb E \big[ \max\limits_{j \in \mathcal J_{T, \de}}\norm{\theta_j - \ol{\theta}_j}^p\big] \right]^{q/p} \, , 
\end{equation}
where with the Gamma function $\,\Gamma(a) \, =\, \int^{\infty}_{0} x^{a-1} e^{-x} {\mathrm d} x \,$, $\, a  > 0\,$, 
%\footnote{(AS) Find some reference for thos formula. Where did you take this from?} %The following term is a constant:
\begin{equation} \label{C1}
C_1 := \left[\mathbb E \Big [ \big( \max\limits_{0 \le t \le T}S(t)\big) ^{qr_1} \Big] \right]^{1/r_1} \, =\,  \frac{\, (2T)^{\, q\, /\, 2}\, }{\, \pi^{1/(2r_{1})}\, } \cdot \Big [\Gamma \Big( \frac{\, 1 + q \, r_{1}\, }{2}\Big)\Big] ^{1\, /\, r_{1}} \,  .
\end{equation}
Here we used the density function of the running maximum of $S$ to compute $\,C_{1}\,$ (see \cite{KaratzasShreve}). 

For the second term in (\ref{eq:X-X}) it is not easy to estimate the number of elements in $\mathcal J_{T, \de}$ directly. Instead, we can estimate it indirectly by the number of elements in a set which is (usually) larger than $\mathcal J_{T, \de}$. Take a number $D (> 0)$ large enough, to be determined later. Define 
$L = (L(t), t \ge 0)$ to be the local time process of the reflected Brownian motion $S$ at zero, i.e.,  $\, L(\cdot) :=  \int^{\cdot}_{0} {\bf 1}_{\{S(t) \, =\, 0\}} {\mathrm d} S(t) \, $. Define also $L^{-1}(s) := \inf\{t \ge 0\mid L(t) = s\}$ to be the inverse local time of $S$. Then the probability $\mathbb P(L(T) > D\sqrt{T})$ is very small. And if $L(T) \le D\sqrt{T}$, then 
\begin{equation}
\label{eq:inclusion}
\mathcal J_{T, \de} \subseteq \mathcal J_{L^{-1}(D\sqrt{T}), \de},
\end{equation}
It follows from \cite[Chapter 12]{RevuzYorBook} that the number $\, \lvert \mathcal J_{L^{-1}(D\sqrt{T}), \de} \rvert\,$ of elements in the set $\mathcal J_{L^{-1}(D\sqrt{T}), \de}$ has Poisson distribution with parameter $\,\la := {D\sqrt{T}}\, / \, {\de}\,$, i.e., 
%Because of~\eqref{eq:Poisson}, we have: 
\begin{equation}
\label{eq:mean}
\mathbb E \big[ |\mathcal J_{L^{-1}(D\sqrt{T}), \de}| \big] = %\la = 
\frac{D\sqrt{T}}{\de}.
\end{equation}
%\begin{equation}
%\label{eq:Poisson}
%|\mathcal J_{L^{-1}(D\sqrt{T}), \de}| \sim \Poi(\la),\ \ \la = \frac{D\sqrt{T}}{\de}.
%\end{equation}
Consider two cases $\, \{L(T) \le D \sqrt{T}\}\,$ and $\,\{L(T) > D \sqrt{T}\}\,$. By~\eqref{eq:inclusion}, 
\begin{equation}
\label{eq:max-to-sum}
\begin{split}
\mathbb E \big[ \max\limits_{j \in \mathcal J_{T, \de}}\norm{\theta_j - \ol{\theta}_j}^p \big] \, =\, \mathbb E \big[ \max\limits_{j \in \mathcal J_{T, \de}}\norm{\theta_j - \ol{\theta}_j}^p \big( {\bf 1}_{\{L(T) \le D \sqrt{T}\} } + {\bf 1}_{\{ L(T) > D \sqrt{T}\}}\big) \big]
\\
\le \mathbb E\big[ \max_{j \in \mathcal J_{L^{-1}(D\sqrt{T}), \de}} \norm{\theta_j - \ol{\theta}_j}^p\big]  + (2\pi)^p \cdot \mathbb P(L(T) > D\sqrt{T}) \, . 
\end{split}
\end{equation}
Since by the construction of the pair $\,(X, \overline{X})\,$ in (\ref{eq: XX}) the random variables $\, \theta_{j}\,$ and $\, \overline{\theta}_{j}\,$ are independent of $\,S\,$, it follows from Wald's identity and (\ref{eq:mean})  that 
\begin{equation} 
\label{eq:Wald}
 \mathbb E \big[ \max_{j \in \mathcal J_{L^{-1}(D\sqrt{T}), \de}}\norm{\theta_j - \ol{\theta}_j}^p \big] \le  \mathbb E \Big[ \!\!\! \SL_{j \in \mathcal J_{L^{-1}(D\sqrt{T}), \de}}\!\!\! \norm{\theta_j - \ol{\theta}_j}^p \Big] = \mathbb E \big[ |\mathcal J_{L^{-1}(D\sqrt{T}) , \de}| \big] \cdot \mathbb E^{\Theta}\big[ \norm{\theta - \ol{\theta}}^p\big] \, .
\end{equation}
%Because of~\eqref{eq:Poisson}, we have: 
%\begin{equation}
%\label{eq:mean}
%\ME|\mathcal J_{L^{-1}(D\sqrt{T}), \de}| = \la = \frac{D\sqrt{T}}{\de}.
%\end{equation}
%Next, $L(T)$ has the same distribution as $S(T)$, the reflected Brownian motion at time $T$ starting from zero. Therefore, if $\Psi(\cdot)$ is the tail of the standard normal distribution, we have: 
The classical L\'evy theorem states that $L(T)$ has the same distribution as $S(T)$. Then by the Mills ratio of the Gaussian tail probability \cite[Chapter 7]{FellerBook}, for every $\,r \ge 1\,$ there exists a constant $\,C_{2}\,$ (which does not depend on $\,D\,$), such that %we obtain the upper bound  
\begin{equation}
\label{eq:estimate}
\mathbb P(L(T) > D\sqrt{T}) = \mathbb P(S(1) > D) %= 2\Psi(D) 
\le \frac{2}{\sqrt{2\pi}D}e^{-D^2/2} \le \frac{\,C_{2}\,}{\,D^{r}\,}%\le e^{-D^2/2} 
\,; \quad \,D > 0\, . 
\end{equation}
Combining (\ref{eq:mean})-(\ref{eq:estimate}) together, we have 
\begin{equation}
\mathbb E \big[ \max\limits_{j \in \mathcal J_{T, \de}}\norm{\theta_j - \ol{\theta}_j}^p \big] \le 
\frac{\, D\sqrt{T}\, }{\de} \,\cdot  \, \mathbb E^{\Theta} \big[ \norm{\theta - \ol{\theta}}^p \big] + \frac{C_{3}}{D^{r}} \, 
\end{equation}
for every $\, D > 0\,$, $\, \delta > 0\,$ where $\,C_{3} \, =\, (2\pi)^{p}C_{2}\,$. Thus the right hand of (\ref{eq:X-X}) is evaluated as 
\begin{equation} \label{eq: new-aggregated}
\mathbb E^{\Pi(\Theta)} \big[ \max\limits_{0 \le t \le T}\norm{X(t) - \ol{X}(t)}^q \big] \le  C_{4} \cdot \Big[ \de^q +  \left(\frac{\, D\, }{\de}\, \cdot \, \mathbb E^{\Theta}\big[ \norm{\theta - \ol{\theta}}^p \big]+ \frac{\,1\,}{\,D^{r}\,}\right) ^{q/p} \Big],
\end{equation}
where $\,C_{4}\,$ does not depend on $\,\delta (> 0), D (> 0) \,$ but on $\,r (\ge 1)\,$, $\, T (> 0)\,$, $\,1 \le q < p \,$.  
%%because of an estimate from \cite[Chapter 7]{FellerBook} and of $D \ge 1$. 
%In sum, combining (\ref{eq:mean})-(\ref{eq:estimate}) together, 
%%~\eqref{eq:estimate},~\eqref{eq:max-to-sum},~\eqref{eq:mean} and~\eqref{eq:Wald}, 
%we have 
%\begin{equation*}
%\label{eq:aggregate}
%\mathbb E \big[ \max\limits_{j \in \mathcal J_{T, \de}}\norm{\theta_j - \ol{\theta}_j}^p \big] \le 
%\frac{\, D\sqrt{T}\, }{\de} \,\cdot  \, \mathbb E^{\Theta} \big[ \norm{\theta - \ol{\theta}}^p \big] + \frac{2(2\pi)^p}{\sqrt{2\pi} D}\, \cdot e^{-D^2/2}\, .  
%\end{equation*}

%Note that for $D \ge D_0 := 2/(2\pi)^q$, we have 
%\begin{equation}
%\label{eq:simple-estimate}
%\frac{D\sqrt{T}}{(2\pi)^q\de} \le \frac{D^2\sqrt{T}}{2\de}.
%\end{equation}
%Therefore, combining~\eqref{eq:X-X},~\eqref{eq:aggregate},~\eqref{eq:simple-estimate}, we have 
%\begin{equation}
%\label{eq:new-aggregated}
%\mathbb E^{\Pi(\Theta)} \big[ \max\limits_{0 \le t \le T}\norm{X(t) - \ol{X}(t)}^q \big] \le  (2\de)^q + C_2\cdot \left(\frac{\, D^2\sqrt{T}\, }{2\de}\, \cdot \, \mathbb E^{\Theta}\big[ \norm{\theta - \ol{\theta}}^p \big]+ e^{-D^2/2}\right)^{q/p},
%\end{equation}
%for $\, \delta > 0\,$, $\,D > 2\,$ and $\, \Theta \in \Gamma_{0}\,$, where $C_2 := (2\pi)^q\, C_1$ with (\ref{C1}) . 

\bigskip 

\noindent {\it Step 3: Minimization.} 
Since the left-hand side of (\ref{eq: new-aggregated}) does not depend on $\,(D, \delta)\,$, let us minimize the right-hand side of (\ref{eq: new-aggregated}) with respect to $\,D\,$ and $\,\delta\,$ by applying twice an inequality 
\begin{equation} \label{eq: harmonic-mean}
f_{0}(x) \, :=\, a_{1} \, x^{c_{1}} + a_{2} \, x^{-c_{2}} \, \ge\,  \Big[ \Big( \frac{\,c_{2}\,}{\,c_{1}\,}\Big)^{\frac{c_{1}}{c_{1}\, +\, c_{2}}} + \Big( \frac{\,c_{1}\,}{\,c_{2}\,}\Big)^{\frac{c_{2}}{c_{1}\,+\, c_{2}}} \Big] \cdot a_{1}^{\frac{c_{2}}{c_{1}\, +\, c_{2}}} \cdot a_{2}^{\frac{c_{1}}{c_{1}\, +\, c_{2}}} \, =\, f_{0}(x^{\ast})
\end{equation}
for every $\, x > 0\,$, where $\,a_{i}, c_{i}\,$, $\,i \, =\, 1, 2\,$ are fixed positive constants and $\,x^{\ast} \, :=\, (\frac{a_{2} \, c_{2}}{a_{1}\,c_{1}})^{1/(c_{1}+c_{2})}\,$ is a unique minimizer of the function $\,f_{0}(\cdot)\,$. Applying (\ref{eq: harmonic-mean}) with $\,(x, a_{1}, a_{2}, c_{1}, c_{2}) = (D, \mathbb E^{\Theta}\big[ \norm{\theta - \ol{\theta}}^p \big]/\delta, 1, 1, r) \,$ and $\,D^{\ast} \, :=\, (\delta r / \mathbb E^{\Theta}\big[ \norm{\theta - \ol{\theta}}^p \big] )^{1/(1+r)} \,$, 
we obtain 
\[
[f_{0}(D)]^{q/p} \, =\, \Big( \frac{\, D\, }{\de}\, \cdot \, \mathbb E^{\Theta}\big[ \norm{\theta - \ol{\theta}}^p \big]+ \frac{\,1\,}{\,D^{r}\,}\Big)^{q/p} \ge C_{5} \cdot \Big( \frac{\,\mathbb E^{\Theta}\big[ \norm{\theta - \ol{\theta}}^p \big] \,}{\,\delta\,} \Big)^{qr/(p(1+r))} \, =\, [f_{0}(D^{\ast})]^{q/p}\, , 
\]
and then applying (\ref{eq: harmonic-mean}) with $\, (x, a_{1}, a_{2}, c_{1}, c_{2}) \, =\, (\delta, 1, C_{5} (\mathbb E^{\Theta}\big[ \norm{\theta - \ol{\theta}}^p \big])^{qr/(p+pr)}, q, qr / ( p ( 1+ r))) \,$, we obtain 
\begin{equation} \label{eq: C6}
C_{4} \cdot \Big( \delta^{q} + C_{5} \cdot \Big( \frac{\,\mathbb E^{\Theta}\big[ \norm{\theta - \ol{\theta}}^p \big] \,}{\,\delta\,} \Big)^{q r\, /\, (p(1+r))}  \Big) \ge C_{6} \big ( \mathbb E^{\Theta}\big[ \norm{\theta - \ol{\theta}}^p \big]  \big)^{\rho \, q/p} \, 
\end{equation}
with some constants $\,C_{i}\,$, $\,i \, =\, 5, 6\,$, where 
\begin{equation} \label{eq: rho}
\rho \, :=\, \frac{p r }{ ( 1 + p ) r + p} %< \frac{\,p\,}{\,1+p\,} \, . 
\, \in \Big[ \frac{\,p\,}{\,2p+1\,}\, , \frac{\,p\,}{\,1+p\,}\Big) \, , \quad \text{ with } \quad r \, =\, \frac{\,p \rho \,}{\, p - (1 + p) \rho \,} \, . 
\end{equation}

Now for every $\, \rho \in [p/(2p+1), p/(1+p))\,$ given, we may choose  the corresponding constants $\, r (\ge 1)\,$ from (\ref{eq: rho}), $\,C_{2}\,$ in (\ref{eq:estimate}) and then the resulting constant $\,C^{\ast} \, :=\, C_{6}\,$ in (\ref{eq: C6}) to  obtain 
\begin{equation} \label{eq:EXXq}
\big [ \mathbb E^{\Pi(\Theta)} \big[ \max\limits_{0 \le t \le T}\norm{X(t) - \ol{X}(t)}^q \big] \big]^{1/q} \, \le\,  C^{\ast} \big ( \mathbb E^{\Theta}\big [ \norm{\theta - \ol{\theta}}^p\big] \big)^{\rho/p}\, . 
\end{equation}
Taking the infimum on both sides of (\ref{eq:EXXq}) over the family  $\,\Gamma_{0}\,$ of measures for the coupling $\, \Theta\,$ on $\,\mathbb S\times \mathbb S\,$, we achieve the desired upper bound~\eqref{eq:main-ineq} for $\, \rho \in [p/(2p+1), p/(1+p))\,$. For every $\, \rho_{1} \in (0, p/(2p+1))\,$ and $\, \rho_{2} \in [p/(2p+1)), p/(1+p)) \,$, because of boundedness of $\, 0 \le W_{p}(\mu, \overline{\mu}) \le \sup_{\theta,  \overline{\theta} \in \mathbb S } \lVert \theta - \overline{\theta} \rVert \, =\, 2\,$, there exists a constant $\,C_{7}\,$ such that $\, [W_{p}(\mu, \overline{\mu})]^{\rho_{2}} \le C_{7}[W_{p}(\mu, \overline{\mu})]^{\rho_{1}} \,$. Using this equality and modifying the constant $\,C^{\ast}\,$ for $\, \rho_{2}\,$, we obtain the desired upper bound ~\eqref{eq:main-ineq} also for $\, \rho \in (0, p/(2p+1))\,$. Therefore, we conclude the proof of ~\eqref{eq:main-ineq} for every $\, \rho \in (0, p/(1+p))\,$.

\hfill $\,\square\,$

\subsection{Proof of Corollary \ref{cor:Walsh-BM-continuous-dependence}} For the case $x_n = \mathbf{0}$ for all $n$, this follows immediately: Convergence in the Wasserstein distance of order $p$ is equivalent to weak convergence plus uniform boundedness of the $p$-th moment, see \cite{Villani08}. For measures $(\mu_n)_{n \ge 0}$ on $\mathbb S$, their $p$-th moments are trivially uniformly bounded, since $\mathbb S$ is a bounded set. Theorem \ref{Thm2.1} makes the rest of the proof trivial.  

Consider the general case of Corollary~\ref{cor:Walsh-BM-continuous-dependence}, with arbitrary initial conditions. By the  Skorohod representation theorem, we can create a probability space with copies $\tilde{W}_n$ of Walsh Brownian motions starting from the origin, such that $ \widetilde{W}_n \to \widetilde{W}_0$ a.s. uniformly on every $[0, T]$. For $n \ge 0$, let $\theta_n := \arg(x_n)$. Then we can couple $W_n$ for $n = 0, 1, 2, \ldots$ as follows. 

Let us take a standard Brownian motion $B = (B(t),\, t \ge 0)$ on the real line with $\, B(0) \, =\, 0\,$. Since $x_n \to x_0$ as $n \to \infty$ in Euclidean topology, for $n$ large enough (and therefore without loss of generality for all $n$) we have: $x_n \ne \mathbf{0}$.
Since $x_n \to x_0$, we have 
%\begin{equation}
%\label{eq:conv-of-arguments}
$\, \norm{x_n} \to \norm{x_0},\ \ \theta_n \to \theta_0\, $.
%\end{equation}

Now, define stopping times $\tau_n := \inf\{t \ge 0\mid B(t) = \norm{x_n}\}$. We construct copies of Walsh Brownian motions $W_n$ starting from $W_n(0) = x_n$, as follows:
\begin{equation}
\label{eq:coupled-new}
W_n(t) = 
\begin{cases}
x_n - \theta_n B(t) = \theta_n(\norm{x_n} - B(t)),\, t \le \tau_n;\\
\tilde{W}_n\left(t - \tau_n\right),\, t \ge \tau_n.
\end{cases}
\end{equation}
Consider the inverse $\mathfrak m^{-1} = (\mathfrak m^{-1}(s),\, s \ge 0)$ of running maximum  of a Brownian motion $B$ starting from zero. This process is a.s. continuous at every fixed time (although it has a.s. discontinuous trajectories). We can express $\tau_n := \mathfrak m^{-1}(\norm{x_n})$. Therefore, a.s. $\tau_n \to \tau_0$ as $n \to \infty$. Applying Lemma \ref{lemma:piecewise-functions} in Appendix, we conclude Corollary \ref{cor:Walsh-BM-continuous-dependence}.

\subsection{Proof of Theorem~\ref{thm:continuous-dependence-on-spinning-measure} in the general case} 

We split the proof into four steps. In the first three steps, we consider the driftless case: $g \equiv 0$. Step 1 is devoted to time-change and localization. That is, we consider an exhaustion of the state space $\mathcal I$ by an increasing sequence of compact domains $D_{m}$, $m \ge 1$. Then we fix one of these domains and stop {\it all} Walsh diffusions $X_{n}$ when they exit the interior of the domain. In Step 2, we prove the convergence result for these stopped Walsh diffusions. In Step 3, we switch from $\Int D_m$ to $\mathcal I$ and  show that for the driftless case the convergence takes place not only for those stopped, but also for the original Walsh diffusions. Finally, in Step 4, we use the scale mapping~\eqref{eq:scale} to extend this result to the general case with non zero drift function $g$. 

\smallskip

\noindent {\it Step 1: Time-change and localization.} Assume $g \equiv 0$. We apply the time-change in section 2.5 to the Walsh diffusions $X_{n}$ on $\mathcal I$ associated with $(0, \sigma, \mu_{n})$: 
\begin{equation} \label{Xnt}
X_n(t) = W_n(T_n(t)),\ \ T_n(t) := \int_0^t\si^2(X_n(s))\,\md s,\ \ t \ge 0,
\end{equation}
where $W_n = (W_n(t),\, t \ge 0)$ is a Walsh Brownian motion with spinning measure $\mu_n$, starting from $W_n(0) = x_n$, $n \ge 0$. By Assumption \ref{asmp:3} and Remark~\ref{rmk:particular-ell}, we may take a sequence of continuous functions $R_n : \mathbb S \to (0, \infty)$ such that $R_n(\theta) \to \ell(\theta)$ pointwise on $\mathbb S$. Recall the definition of $D_m$ in~\eqref{eq:D-m}. For $n \ge 0$ and $m \ge 1$, define the stopping times 
$$
\tau_n^{(m)} := \inf\{t \ge 0\mid X_n(t) \notin D_m\} \, =\, \inf \{ t \ge 0 \mid W_{n} (T_{n}(t)) \not \in D_{m}\} , 
$$
and the corresponding stopped processes $ X^{(m)}_n(t) := X_n (t\wedge\tau^{(m)}_n )$, $ T_n^{(m)}(t) := T_n(t\wedge\tau^{(m)}_n)$, $ t \ge 0$. 
Then it follows from (\ref{Xnt}) that for $n \ge 0$ and $m \ge 1$, 
\begin{equation}
\label{eq:connection}
X_n^{(m)}(t) = W_n\left(T^{(m)}_n(t)\right)  , \quad t \ge 0 . 
\end{equation}
In particular, $X_{0}^{(m)} (t) = W_{0}(T_{0}^{(m)}(t))$, $t \ge 0$. 

\medskip

\noindent {\it Step 2: Proof for the stopped processes.} The rest of the proof for the driftless case is quite similar to the proof of \cite[Theorem 2.2]{MyOwn8}. For the rest of Step 2, fix an $m \ge 1$. We shall show 

\begin{lemma}  Every subsequence of $(X_n^{(m)})_{n \ge 1}$ in (\ref{eq:connection}) has a weakly convergent subsequence, and this weak limit behaves as the Walsh (driftless) diffusion $X_0^{(m)}$, at least as long as it stays in $\Int D_m$.
\label{lemma:weak-limit-points}
\end{lemma}

The rest of Step 2 is devoted to  proving Lemma~\ref{lemma:weak-limit-points}. For every $n \ge 0$ and $t \ge 0$, we have 
\begin{equation}
\label{eq:bound-on-sigma}
0 \le \si\left(X_n^{(m)}(t)\right) \le \max\limits_{D_m}\si =: \ol{\si}_m, 
\end{equation}
and hence, by definition of $T_{n}(t)$ in (\ref{Xnt}), 
\begin{equation}
\label{eq:bound-on-T}
0 \le T_n^{(m)}(t) \le \ol{\si}^2_mt.
\end{equation}
By the Arzela-Ascoli criterion combined with~\eqref{eq:bound-on-sigma}, the sequence $(T_n^{(m)})_{n \ge 1}$ is tight for every $m \ge 1$. Then every subsequence $(n_k)_{k \ge 1}$ of $\mathbb N$ has its further subsequence $(n'_k)_{k \ge 1}$ for which there exists a random process $\ol{T}^{(m)} = (\ol{T}^{(m)}(t),\, t \ge 0)$ such that 
\begin{equation} 
\label{eq:convergence-of-T}
T_{n'_k}^{(m)} \to \ol{T}^{(m)}\ \mbox{in}\ \ C[0, T]\,,\ \ \mbox{as}\ \ k \to \infty. 
\end{equation}

As we have already proved (in the proof of Corollary~\ref{cor:Walsh-BM-continuous-dependence}), we have $W_n \Ra W_0$ in $C([0, \ol{\si}_m^2T], \BR^d)$, as $\mu_{n}\Ra \mu_{0}$ and $x_{n}\Ra x_{0}$. Changing the probability space, if necessary, by the virtue of Skorohod representation theorem, we can make this convergence a.s. 
\begin{equation}
\label{eq:convergence-of-W}
W_n \to W_0\ \ \mbox{in}\ \ C([0, \ol{\si}^2_mT], \BR^d).
\end{equation}
Combining~\eqref{eq:connection},~\eqref{eq:bound-on-T}, ~\eqref{eq:convergence-of-T} and ~\eqref{eq:convergence-of-W}, we have the limit 
\begin{equation}
\label{eq:convergence-of-X}
X_{n'_k}^{(m)}(t) \to \ol{X}^{(m)}(t) := W_0(\ol{T}^{(m)}(t))\ \ \mbox{uniformly on}\ \ [0, T]. 
\end{equation}
Since the function $\si$ is continuous on $D_m$ by Assumption \ref{asmp:3}, we also get  
\begin{equation}
\label{eq:convergence-of-sigma}
\si\left(X_{n'_k}^{(m)}(t)\right) \to \si(\ol{X}^{(m)}(t)),\ \ k \to \infty,\ \ \mbox{for every}\ \ t \in [0, T].
\end{equation}
By~\eqref{eq:bound-on-sigma},~\eqref{eq:convergence-of-sigma}, and the Lebesgue dominated convergence theorem, for every $t \in [0, T]$, we have  
\begin{equation}
\label{eq:convergence-of-integrals}
\int_0^t\si^2\left(X_{n'_k}^{(m)}(s)\right)\,\md s \to \int_0^t\si^2\left(\ol{X}^{(m)}(s)\right)\,\md s.
\end{equation}

Denote $\tau_{\infty}^{(m)} := \varliminf_{n \to \infty}\tau^{(m)}_n$, and take a (random) time point $t_{0} < \tau_{\infty}^{(m)} \wedge T$. For every sufficiently large $k$, and all $s \in [0, t_{0}] \subseteq [0, \tau_{\infty}^{(m)}]$, we  have $X_{n'_k}^{(m)}(s) = X_{n'_k}(s)$ and $t_{0} \wedge \tau_{n_{k}^{\prime}}^{(m)} < \tau_{\infty}^{(m)} \wedge T$. Thus, combining~\eqref{Xnt}-(\ref{eq:connection}) with ~\eqref{eq:convergence-of-integrals}, for such $t_{0}$ we have  
\begin{equation*}
%\label{eq:convergence-of-T-new}
T_{n'_k}^{(m)}(t_{0}) = T_{n_{k}^{\prime}} (t_{0} \wedge \tau_{n_{k}^{\prime}}^{(m)})  = \int^{t_{0} \wedge \tau_{n_{k}^{\prime}}^{(m)} }_{0} \sigma^{2}( X_{n_{k}^{\prime}} (s)) {\mathrm d} s  = \int^{t_{0} \wedge \tau_{n_{k}^{\prime}}^{(m)} }_{0} \sigma^{2}( X_{n_{k}^{\prime}}^{(m)} (s)) {\mathrm d} s  \to \int^{t_{0}}_{0} \sigma^{2}( \overline{X}^{(m)} (s)) {\mathrm d} s  %\ol{T}^{(m)}(t),\ k \to \infty.
\end{equation*}
as $k\to \infty$. Comparing  this observation %~\eqref{eq:convergence-of-T-new} 
with~\eqref{eq:convergence-of-T}, we get  
\begin{equation}
\label{eq:new-relation}
\ol{T}^{(m)}(t) = \int_0^t\si^2(\ol{X}^{(m)}(s))\,\md s, \ \ \mbox{and}\ \ \ol{X}^{(m)}(t) = W_0(\ol{T}^{(m)}(t))\ \ \mbox{for}\ \ t < \tau_{\infty}^{(m)}.
\end{equation}
% This, in turn, means that $$. The last equality implies that the process 
The system ~\eqref{eq:new-relation} of equations implies that every subsequence 
$(n_k)_{k \ge 1}$ has another subsequence $(n'_k)_{k \ge 1}$ such that $X^{(m)}_{n'_k} \Ra \ol{X}$ in $C([0, T], \BR^d)$, where, at least until $\tau_{\infty}^{(m)}$, the process $\ol{X}$ behaves as a Walsh diffusion starting from $\ol{X}(0) = x_0$, associated with $(0, \si, \mu_0)$. For $\, m \ge 1\,$ define %a stopping time 
$$
\ol{\tau}^{(m)} := \inf\{t \ge 0\mid \ol{X}^{(m)}(t) \notin \Int D_m\} = \inf \{ t \ge 0 \mid W_{0}( \overline{T}^{(m)}(t) ) \not \in \text{int} D_{m} \} .
$$
Then we claim the following inequality 
%\begin{equation}
%\label{eq:comparison-of-moments}
$\, \ol{\tau}^{(m)} \le \tau_{\infty}^{(m)}\ \ \mbox{a.s.} \,$
%\end{equation}
for every $\, m \ge 1\,$. Indeed, assume the converse, i.e., $ \overline{\tau}^{(m)} > \tau_{\infty}^{(m)}$ with positive probability. Then there would be  a positive random variable $t_{1} < \ol{\tau}^{(m)}$ such that there exists a sequence $(\tilde{n}_k)_{k \ge 1}$ with $\tau^{(m)}_{\tilde{n}_k} \le t_{1}$ for all $k$. Since $X^{(m)}_{\tilde{n}_k}$ is stopped at $\tau^{(m)}_{\tilde{n}_k}$, we have  
$\, 
X^{(m)}_{\tilde{n}_k}\left(\tau^{(m)}_{\tilde{n}_k}\right) = X^{(m)}_{\tilde{n}_k}(t_{1}) \in \pa D_m.
\, $
Since $\pa D_m$ is closed, letting $k \to \infty$, we would have $\ol{X}(t_{1}) \in \pa D_m$ with positive probability. This, however, contradicts the property $t_{1} < \ol{\tau}^{(m)}$, which completes the proof of $\,\ol{\tau}^{(m)} \le \tau_{\infty}^{(m)}\ \ \mbox{a.s.}\,$%~\eqref{eq:comparison-of-moments}. 

This proves that $\ol{X}$ behaves as a Walsh diffusion starting from $\ol{X}(0) = x_0$, associated with $(0, \si, \mu_0)$, at least until it exits $\Int D_m$. This completes the proof of Lemma~\ref{lemma:weak-limit-points}. 

\medskip

\noindent {\it Step 3: Proof for the driftless case.} This step is similar to \cite[Lemma 3.8]{MyOwn8}. For a given $\eta \in (0, 1) $ we may take an $m$ large enough, so that the set $\mathcal A := \left\{x \in C([0, T], \BR^d)\mid x(t) \in \Int D_m\ \ \forall\ t \in [0, T]\right\}$ has probability greater than $1 - \eta$, i.e., $ \mathbb P \left(X_0 \in \mathcal A\right)  > 1 - \eta$, since $\mathbb P  \left(X_0(t) \in \mathcal I\ \ \forall \ t \in [0, T]\right) = 1$, and $ \Int D_m \uparrow \mathcal I$ as $\, m\to \infty\,$.

Now, by Lemma \ref{lemma:weak-limit-points} every sequence $(n_k)_{k \ge 1}$ has a subsequence $(n'_k)_{k \ge 1}$, such that $X_{n'_k}^{(m)} \Ra \ol{X}$ in $C([0, T], \BR^d)$. Then for every Borel subset $\mathcal B \subseteq C([0, T], \BR^d)$, we have 
\begin{equation}
\label{eq:equality-in-law}
\mathbb P  \left(\ol{X} \in \mathcal A\cap\mathcal B\right) = \mathbb P  \left(X_0 \in \mathcal A\cap\mathcal B\right).
\end{equation}
In particular, letting $\mathcal B := C([0, T], \BR^d)$, we have:
\begin{equation}
\label{eq:eta-estimate}
\mathbb P  \left(\ol{X} \in \mathcal A\right) = \mathbb P  \left(X_0 \in \mathcal A\right) > 1 - \eta.
\end{equation}
Moreover, for every Borel subset $\mathcal B \subseteq C([0, T], \BR^d)$ and $n  \ge 1$, we have:
\begin{equation}
\label{eq:correct-stopping}
\mathbb P  \left(X^{(m)}_n \in \mathcal A\cap\mathcal B\right) = \mathbb P  \left(X_n \in \mathcal A\cap\mathcal B\right).
\end{equation}
For an open subset $\mathcal G \subseteq C([0, T], \BR^d)$, the subset $\mathcal A\cap\mathcal G$ is also open, and hence, 
\begin{equation}
\label{eq:portmanteau12}
\mathbb P  (X_0 \in \mathcal A\cap\mathcal G) = \mathbb P  (\ol{X} \in \mathcal A\cap\mathcal G) \le \varliminf_{k \to \infty} \mathbb P  \left(X^{(m)}_{n^{\prime}_k} \in \mathcal A\cap \mathcal G\right).
\end{equation}

On the other hand, it follows from~\eqref{eq:equality-in-law} and~\eqref{eq:eta-estimate} that 
\begin{equation}
\begin{split}
\mathbb P  &\left(\ol{X} \in \mathcal A\cap\mathcal G\right) = \mathbb P  \left(X_0 \in \mathcal A\cap\mathcal G\right)  \ge \mathbb P  \left(X_0 \in \mathcal G\right) - \mathbb P  \left(X_0 \notin \mathcal A\right) \ge \mathbb P  (X_0 \in \mathcal G) - \eta.
\label{eq:1205}
\end{split}
\end{equation} 
Combining~\eqref{eq:correct-stopping},~\eqref{eq:portmanteau12} with ~\eqref{eq:1205}, we obtain 
\begin{equation}
\label{eq:new-conv}
\varliminf\limits_{k \to \infty}\mathbb P\left(X_{n'_k} \in \mathcal G\right)  \ge \varliminf\limits_{k \to \infty}\mathbb P\left(X^{(m)}_{n'_k} \in \mathcal A\cap\mathcal G\right) \ge \mathbb P\left(X_0 \in \mathcal G\right) - \eta.
\end{equation}
Thus for every sequence $(n_k)_{k \ge 1}$ and every $\eta > 0$ there exists a subsequence $(n'_k)_{k \ge 1}$ such that~\eqref{eq:new-conv} holds. Use the diagonal argument: let $n^{(0)}_k := n_k$, and construct $(n^{(l)}_k)_{k \ge 1}$ inductively: 
$(n^{(l)}_k)_{k \ge 1}$ plays the role of $(n'_k)_{k \ge 1}$ for $n_k := n^{(l-1)}_k$, $\eta := l^{-1}$. Then for $\ol{n}_k := n^{(k)}_k$, we have 
$\, 
\varliminf\limits_{k \to \infty}\mathbb P\left(X_{\ol{n}_k} \in \mathcal G\right)  \ge \mathbb P\left(X_0 \in \mathcal G\right).
\, $
Therefore, we claim that for every sequence $(n_k)$, there exists a subsequence $(\ol{n}_k)$ such that $X_{\ol{n}_k} \Ra X_0$ in $C([0, T], \BR^d)$. This completes the proof of Theorem~\ref{thm:continuous-dependence-on-spinning-measure} for the case $g \equiv 0$. 

\medskip

\noindent {\it Step 4: General case.} The general case (with an arbitrary drift function $g$) can be reduced to the driftless case by scale transformation~\eqref{eq:scale-mapping}. Recall that the process $\mathcal P(X_n(\cdot))$ is a Walsh diffusion starting from $\mathcal P(x_n)$, associated with $(0, \tilde{\si}, \mu_n)$, where  $\tilde{\mathcal I}$ and $\tilde{\si}$ are given by~\eqref{eq:scale-mapping} and~\eqref{eq:new-sigma}. Lemma~\ref{lemma:scale-continuous} below, together with continuity of $\si$, implies continuity of the function $\tilde{\si}$, and of the mappings $\mathcal P$ and $\mathcal P^{-1}$. Since $x_n \to x_0$, we have $\mathcal P(x_n)  \to \mathcal P(x_0)$. Moreover, because of $\mu_n \Ra \mu_0$, from results just proven, we have $\mathcal P(X_n(\cdot)) \Ra \mathcal P(X_0(\cdot))$ in $C([0, T], \BR^d)$. Finally, because $\mathcal P^{-1}$ is continuous, we have $X_n \Ra X_0$ in $C([0, T], \BR^d)$. This completes the proof of Theorem~\ref{thm:continuous-dependence-on-spinning-measure} for the general case, given Lemma \ref{lemma:scale-continuous} below. 

\begin{lemma}
\label{lemma:scale-continuous}
Under Assumption \ref{asmp:3}, the scale function $s(r, \theta)$ from~\eqref{eq:scale}, and the inverse scale function $s^{-1}(r, \theta)$ from~\eqref{eq:inverse-scale}, are continuous in the  Euclidean topology. 
\end{lemma}

\begin{proof} Under Assumption \ref{asmp:3}, continuity of the scale function $s(\cdot)$ follows from continuity and local boundedness of the function $2g\si^{-2}$, together with Lebesgue dominated convergence theorem. Let us take a sequence $(r_{n}, \theta_{n}) $ that converges to $(r_{0}, \theta_{0})$, i.e., $(r_n, \theta_n) \to (r_0, \theta_0)$. We shall show that $s^{-1}(r_n, \theta_n) \to s^{-1}(r_0, \theta_0)$, that is, for every $\eps > 0$ and for all $n$, except finitely many, we have  
\begin{equation}
\label{eq:comparison-of-scales}
s^{-1}(r_n, \theta_n) < s^{-1}(r_0, \theta_0) + \eps.
\end{equation} 
Since $s(r, \theta)$ is strictly increasing in $r$ for every fixed $\theta$, ~\eqref{eq:comparison-of-scales} would be equivalent to 
\begin{equation}
\label{eq:333}
r_n < s\left(s^{-1}(r_0, \theta_0) + \eps, \theta_n\right).
\end{equation}
Note that letting $n \to \infty$ in~\eqref{eq:333} and using continuity of $s$, we see the left-hand side converges to $r_0$, and the right-hand side converges to $s\left(s^{-1}(r_0, \theta_0) + \eps, \theta_0\right)$. Here since 
$$
r_0 = s\left(s^{-1}(r_0, \theta_0), \theta_0\right) < s\left(s^{-1}(r_0, \theta_0) + \eps, \theta_0\right),
$$
we must get~\eqref{eq:333} for large enough $n$. This completes the proof of Lemma~\ref{lemma:scale-continuous}. 
\end{proof}

\section{Feller and Positivity Properties} 

In this section, we study several properties: Feller property; positivity of transition kernel (that $\mes\otimes\mu$-positive subsets of $\BR^d$ have positive transition measure). They are necessary for the next section, where we find Lyapunov functions for Walsh diffusions to prove  existence and uniqueness of a stationary distribution and convergence to this stationary distribution as $t \to \infty$. 

\subsection{Positivity of transition kernel} 
Let us first mention intuition for general Markov processes. Very loosely speaking, a Markov process is called {\it irreducible}, if the state space cannot be separated into two or more parts such that the process cannot move between them; and it is called {\it aperiodic} if the state space cannot be separated into two or more parts such that the process circulates between them.  If the process is tight, then irreducibility and aperiodicity guarantee existence and uniqueness of a probability invariant measure. Since we do not use these particular (very important) concepts in this paper, we shall not rigorously define them here. But we, however, prove a stronger property: positivity. 
%\footnote{(TI) I am not completely sure about the importance of this paragraph. Can we rephrase? (AS) Yes, it would be nice. Or just remove it. I don't care. }
%Now, let us return to Walsh diffusions. 
Let us define the following {\it reference measure}:
\begin{equation}
\label{eq:reference-measure}
\tilde{\mu} := \mu\, \otimes\, \mes\, 
\end{equation}
and consider a Walsh diffusion on $\, \mathcal I\,$ associated with $(g, \si, \mu)$ and the transition kernel $\, P^t(x, \cdot) = \mathbb P(X(t) \in  \cdot \mid X(0) = x)\, $ for $\, x \in \mathcal  I\,$, $\, t >  0\,$. 
%Let $P^t(x, C) = \MP(X(t) \in C\mid X(0) = x)$ be the transition kernel of this Walsh diffusion. The following property is called {\it positivity} (of transition kernel); it is stronger than irreducibility and aperiodicity combined. 

\begin{thm} Under Assumptions \ref{asmp:1} and \ref{asmp:2}, the transition kernel is positive, that is, for every $t > 0$, $x \in \mathcal I$, and a Borel subset $C \subseteq \mathcal I$ with $\tilde{\mu}(C) > 0$,  we have $\, P^t(x, C) > 0\, $. 
\label{thm:positivity}
\end{thm}

\begin{proof} We shall show this theorem in four steps. 

\medskip 

{\it Step 1.} Let us argue first that it suffices to show the case $x = \mathbf{0}$. Indeed, if the initial value $ X(0) = x = (r, \theta) \in \mathcal I $ is not the origin $\, \mathbf{0}\,$, then until the first hitting time $\tau_{\mathbf{0}} := \inf\{t \ge 0\mid X(t) = \mathbf{0}\}$, $\,X(\cdot)\,$ can be represented as $X(\cdot) = \theta Z(\cdot)$, where $Z = (Z(t),\, t \ge 0)$ is a diffusion on the half-line with drift $g(\cdot, \theta)$ and diffusion $\si^2(\cdot, \theta)$, starting from $Z(0) = r > 0$ and killed at the origin. Then $\tau_{\mathbf{0}} = \inf\{t \ge 0\mid Z(t) = 0\}\, $ and hence for such diffusion $\,Z(\cdot)\,$ we have $\mathbb P(\tau_{\mathbf{0}} < t) > 0$, $\, t > 0\,$, since  the functions $g, \si, \si^{-1}$ are locally bounded on $\BR_+$. Thus if the statement of the theorem is true in the case $\, x \, =\, \mathbf 0\,$, then  $\, P^{u} ( {\mathbf 0}, C) > 0\,$ for every $\,u > 0\,$, and hence, 
$$
P^t(x, C) \ge \int_0^tP^{t-s}( {\mathbf 0}, C)\,\mathbb P(\tau_{\mathbf{0}} \in \md s \vert X(0) = x) > 0\, ; \quad x \in I\, ,  
$$
because the Lebesgue integral of a positive function over a set of positive measure is positive, and hence the statement is true for every $\, x \in I\,$. Thus it suffices to show the case $\, x \, =\, \mathbf 0\,$. 

\medskip

{\it Step 2.} Next, consider the case of Walsh Brownian motion $\, X(\cdot) \, :=\, W(\cdot) \,$ associated with $g \equiv 0$ and $\si \equiv 1$ starting at the origin, with  the set $\, C\,$ of the form $C = A \times B \subseteq \mathbb S \times (0, \infty) $ with $\, \tilde{\mu} (C) >  0\,$. Then by the construction (e.g., \cite[Theorem 2.1]{IchibaNew}), $\arg (W(t))$ is distributed as $\, \mu\,$, independent of $\norm{W(t)}$, a reflected Brownian motion on the half-line, starting from zero. Thus %Therefore, for every $\, t > 0\,$, 
$$
P^t( \mathbf 0, C) = \mathbb P(\arg (W(t)) \in A,\, \norm{W(t)} \in B) = \mu(A)\cdot\mathbb P( \norm{W(t)} \in B) > 0\, ; \quad t > 0 \, . 
$$

\medskip

{\it Step 3.} Now consider the case of {\it driftless} Walsh diffusions with $g \equiv 0$. From \cite[Proposition 3.4]{MinghanNew}, we have the Dambis-Dubins-Schwartz-type representation: $X(t) = W(T(t))$, where $W  (\cdot )$ is a Walsh Brownian motion in Step 2 starting from $W(0) = \mathbf 0$, and 
\begin{equation}
\label{eq:expression-for-time-change}
T(t) = \int_0^t\si^2(X(s))\,\md s , \quad t \ge 0 \, . 
\end{equation}
For the set $C = A \times B$ define $R := \sup B < l_{\min} := \inf_{\theta \in A} \ell(\theta)$, and for a fixed $R' \in (R,  l_{\min})$ 
\begin{equation}
\label{eq:stopping-times-1}
\tau_{R, A}^{X} := \inf\{t \ge 0\mid \norm{X(t)} = R,\, \arg X(t) \in A\},
\end{equation}
\begin{equation}
\label{eq:stopping-times-2}
\tau_{R, A}^{W} := \inf\{t \ge 0\mid \norm{W(t)} = R,\, \arg W(t) \in A\},
\end{equation}
\begin{equation}
\label{eq:stopping-times-3}
\tau_{R'}^X := \inf\{t \ge 0\mid \norm{X(t)} = R'\}, \quad 
%\end{equation}
%\begin{equation}
%\label{eq:stopping-times-4}
\tau_{R'}^W := \inf\{s \ge 0\mid \norm{W(s)} = R'\} \, . 
\end{equation}

With Lemma \ref{lemma:intermediate-positivity} in Appendix we claim $\mathbb P \left(\tau^X_{R, A} < t\right) > 0$. Note that 
%Let us complete Step 3 of the proof of Theorem~\ref{thm:positivity} for the driftless Walsh diffusions. 
\begin{equation}
\label{eq:final-estimate}
\mathbb P(X(t) \in C = A\times B) \ge \int_0^t\int_Ap(t-s, \theta)\,\mathbb P(\tau^X_{R, A} \in \md s,\,\arg (X(s)) \in \md\theta),
\end{equation}
where $p(u, \theta)$ is the probability that a reflected diffusion on the half-line with zero drift and diffusion $\si^2(\cdot, \theta)$, starting from $R$, stays in $(0, R']$ on the time interval $[0, u]$, and hits the set $B$ at time $u$. From boundedness of $\si$ and $\si^{-1}$ on $[0, R']$, which follows from Assumption \ref{asmp:1}, we have: $p(u, \theta) > 0$. Again using the observation that the Lebesgue integral of a positive function over a set of positive measure is positive, we see that the right hand of~\eqref{eq:final-estimate} is positive, which completes the proof of Theorem~\ref{thm:positivity} for the case of driftless Walsh diffusions.

\medskip

{\it Step 4.} Finally, let us prove Theorem~\ref{thm:positivity} in the general case. It can be reduced via \cite[Proposition 3.12]{MinghanNew} to the driftless case. Using the notation from there, we observe that the one-to-one function $\mathcal P : \mathcal I \to \tilde{\mathcal I}$ from~\eqref{eq:scale-mapping} maps the Walsh diffusion with nonzero drift to another Walsh diffusion with zero drift. The new Walsh diffusion also has a diffusion coefficient $\tilde{\si}$ from~\eqref{eq:new-sigma} such that $\tilde{\si}$ and $\tilde{\si}^{-1}$ are both locally bounded on $\tilde{\mathcal I}$. Also, the map $\mathcal P$, as well as its inverse $\mathcal Q$, maps $\tilde{\mu}$-positive subsets into $\tilde{\mu}$-positive subsets. This follows from the observation that these maps preserve arguments of points: $\arg(\mathcal P(x)) = \arg (x)$ for $x \in \mathcal I\setminus\{\mathbf{0}\}$, and the radial derivative of $\mathcal P$ is everywhere positive. Thus Theorem~\ref{thm:positivity} is a simple corollary for the driftless case. 
\end{proof}

\subsection{Feller property for the tree topology} Next, we shall prove the Feller property of Walsh diffusions, that is, the semigroup maps bounded continuous functions into bounded continuous functions. Fix a bounded continuous (in the tree topology) function $f : \mathcal I \to \BR$.  

\begin{thm} 
\label{thm:Feller} Under Assumptions \ref{asmp:1} and \ref{asmp:2}, for $t > 0$, if $x \to x_0$ in $\mathcal I$, then 
$$
\mathbb E_x [f(X(t)) ] \to \mathbb E_{x_0}[ f(X(t))]\, .
$$
\end{thm}

\begin{proof} We shall first consider the harder case $x_0 \ne \mathbf{0}$, and then discuss  the easier case.

\medskip

\noindent {\it Case 1.} Assume $x_0 = \theta r_0 \ne \mathbf{0}$ for some $\theta$, that is, $r_0 > 0$; and $x = \theta r$, with $r \uparrow r_0$. Take a copy $X^{(r_0)} = (X^{(r_0)}(t),\, t \ge 0)$ of this Walsh diffusion starting from $X^{(r_0)}(0) = x_0$. One can construct (on the same probability space as $X^{(r_0)}$) a family $(Z^{(r)})_{r \in (0, r_0]}$ of reflected diffusions on the half-line with drift $g(\cdot, \theta)$ and diffusion $\si^2(\cdot, \theta)$, starting from $Z^{(r)}(0) = r$, such that $Z^{(r)}(t) \le Z^{(r')}(t)$ a.s. for every $t \ge 0$ and $0 < r < r' \le r_0$. Also, assume that this probability space contains a Walsh diffusion $X^{(0)} = (X^{(0)}(t),\, t \ge 0)$ with the same drift and diffusion coefficients, starting from the origin, independent of everything else. 
%\smallskip
Let $\tau_{r, a} := \inf\{t \ge 0\mid Z^{(r)}(t) = a\}$ for $r \in (0, r_0]$ and $a \in \BR_+$. Then since 
%\begin{equation}
%\label{eq:convergence-of-hitting-times}
$ \tau_{r, r_0} \downarrow \tau_{r_0, r_0} = 0$, $ \tau_{r, 0} \uparrow \tau_{r_0, 0} > 0 $ a.s., as $r \uparrow r_0$, 

%%-
%\newpage
%%-

%\end{equation}
%Note that it follows from~\eqref{eq:convergence-of-hitting-times} that as $r \uparrow r_0$, 
\begin{equation}
\label{eq:convergence-of-indicators}
\begin{split}
\lim_{r\uparrow r_{0}} 1\left(\tau_{r, r_0} < \tau_{r, 0}\right) = 1 , \quad  \lim_{r\uparrow r_{0}}  1\left(\tau_{r, r_0} > \tau_{r, 0}\right) = 0 , \\ %\to 0 \ \ \mbox{a.s.}\ \ 
\lim_{r\uparrow r_{0}}  1\left(t \le \tau_{r, r_0} < \tau_{r, 0}\right) = 0,\quad \lim_{r\uparrow r_{0}} 1\left(\tau_{r, r_0} < \tau_{r, 0}\wedge t\right) = 1\ \ \mbox{a.s.} 
%\ \ r \uparrow r_{0}.
\end{split}
\end{equation}
%and  for every $t > 0 $
%Because of~\eqref{eq:convergence-of-hitting-times}, we have: 
%\begin{equation}
%\label{eq:new-conv-of-indicators}
%1\left(t \le \tau_{r, r_0} < \tau_{r, 0}\right) \to 0,\ \ 1\left(\tau_{r, r_0} < \tau_{r, 0}\wedge t\right) \to 1\ \ \mbox{a.s. as}\ \ r \uparrow r_{0}.
%\end{equation}

For every $r \in (0, r_0)$, let us construct a copy $X^{(r)} = (X^{(r)}(t),\, t \ge 0)$ of this Walsh diffusion starting from $X^{(r)}(0) = \theta r$. 
The construction of $X^{(r)}$ proceeds as follows.
\smallskip

\noindent (a) If $\tau_{r, 0} < \tau_{r, r_0}$, that is, the reflected diffusion $Z^{(r)}$ hits zero before $r_0$, then we let $X^{(r)}$ evolve like this reflected diffusion on the ray $\mathcal R_{\theta}$ before hitting the origin, and then start the independent copy $X^{(0)}$ of the Walsh diffusion from there. Formally, let us define 
$$
X^{(r)}(t) := 
\begin{cases}
\theta Z^{(r)}(t),\, t \le \tau_{r, 0};\\
X^{(0)}\left(t - \tau_{r, 0}\right),\, t \ge \tau_{r, 0}.
\end{cases}
$$

\noindent (b) If $\tau_{r, 0} > \tau_{r, r_0}$, that is, $Z^{(r)}$ hits $r_0$ before zero, then we let $X^{(r)}$ evolve like this reflected diffusion on the ray $ \mathcal R_{\theta}$, until it hits $x_0$. Then we start the copy of $X^{(r_0)}$. Let us define %This is the most likely case as $r \uparrow r_0$. 
%First, let us define formally 
$$
X^{(r)}(t) := 
\begin{cases}
\theta Z^{(r)}(t),\, t \le \tau_{r, r_0};\\
X^{(r_0)}\left(t - \tau_{r, r_0}\right),\, t \ge \tau_{r, r_0}. 
\end{cases}
$$
Since $r \uparrow r_0$ case (a) is less likely and case (b) is more likely. 
Thus we construct a Walsh diffusion $X^{(r)}$ with initial value $X^{(r)}(0) = (r, \theta) $. We shall evaluate $\,\mathbb E\left[f(X^{(r)}(t))\right] \, =: \,  E(r) + F(r) \,$, where 
\begin{align}
\label{eq:split-into-two}
F(r)  := \mathbb E\left[f(X^{(r)}(t))1\left(\tau_{r, r_0} < \tau_{r, 0}\right)\right] , \quad E(r) := \mathbb E\left[f(X^{(r)}(t))1\left(\tau_{r, r_0} > \tau_{r, 0}\right)\right]  .
\end{align}

Thanks to ~\eqref{eq:convergence-of-indicators} and boundedness of $f$, we immediately see  
\begin{equation}
\label{eq:E(r)}
E(r) \to 0\ \mbox{as}\ r \uparrow r_0\, .
\end{equation}
Let us decompose the term $F(r)$ into two terms: 
\begin{equation}
\label{eq:F(r)}
\begin{split}
F(r) & = \mathbb E\left[f(X^{(r)}(t))1\left(t \le \tau_{r, r_0} < \tau_{r, 0}\right)\right] + 
\mathbb E\left[f(X^{(r)}(t))1\left(\tau_{r, r_0} < \tau_{r, 0}\wedge t\right)\right] \\ & = 
\mathbb E\left[f(\theta Z^{(r)}(t))1\left(t \le \tau_{r, r_0} < \tau_{r, 0}\right)\right] + \mathbb E\left[f(X^{(r_0)}(t - \tau_{r, r_0}))1\left(\tau_{r, r_0} < \tau_{r, 0}\wedge t\right)\right] \\ & =: F_1(r) + F_2(r) \, .
\end{split}
\end{equation}
Next, $F_1(r) \to 0$ as $r \uparrow r_0$, thanks to ~\eqref{eq:convergence-of-indicators} and boundedness of $f$, and 
\begin{equation}
\label{eq:F-2}
F_2(r) - \mathbb E\left[f\left(X^{(r_0)}\left(t - \tau_{r, r_0}\right)\right)\right] \to 0\ \ \mbox{as}\ \ r \uparrow r_0.
\end{equation}
Combine~\eqref{eq:convergence-of-indicators} with a.s. continuity of trajectories of $X^{(r_0)}$, continuity and boundedness of $f$, and use the Lebesgue dominated convergence theorem. Then, as $r \uparrow r_0$, we get 
\begin{equation}
\label{eq:final-conv}
\mathbb E\left[f\left(X^{(r_0)}\left(t - \tau_{r, r_0}\right)\right)\right] \to \mathbb E\left[f\left(X^{(r_0)}(t)\right)\right].
\end{equation}
Combining~\eqref{eq:split-into-two}-%,~\eqref{eq:E(r)},~\eqref{eq:F(r)},~\eqref{eq:F-2},~
\eqref{eq:final-conv}, we conclude  $\lim_{r \uparrow r_{0}} \mathbb E\left[f(X^{(r)}(t))\right] = \mathbb E\left[f(X^{(r_0)}(t))\right]$.%, which completes the proof of Theorem~\ref{thm:Feller} in Case 1.

\medskip

\noindent {\it Case 2.} Assume $x_0 = \theta r_0 \ne \mathbf{0}$, that is, $r_0 > 0$; and $x = \theta r$, with $r \downarrow r_0$; or $x_0 = \mathbf{0}$, that is, $r_0 = 0$; then also we have $x = \theta r$ with $r \downarrow r_0 = 0$. Then the proof is simpler: there is no case (a). Indeed, a reflected diffusion $Z^{(r)}$ on the half-line, starting from $r > r_0$, must hit $r_0$ before hitting zero (or at least at the same time when $r_0 = 0$). The details of the proof are left to the reader. 
%\footnote{Let us delete this last sentence, but do we need to write any more details?}
\end{proof}

\subsection{Feller property for Euclidean topology} The continuity of the transition kernel in Euclidean topology can be seen as a  corollary of Theorem~\ref{thm:continuous-dependence-on-spinning-measure}, if we take $\mu_n \equiv \mu$ for all $n = 0, 1, 2, \ldots$ %We need an additional Assumption 3 from Section 3.

\begin{lemma}
\label{lemma:Feller-Euclidean}
Under Assumptions \ref{asmp:1}, \ref{asmp:2}, \ref{asmp:3}, take a bounded continuous (in the Euclidean topology) function $f : \mathcal I \to \BR$. Fix a $t > 0$, and let $x \to x_0$ in $\mathcal I$ in Euclidean topology. Then 
$$
\mathbb E_{x}[ f(X(t))]  \to \mathbb E_{x_0}[ f(X(t))] \, .
$$
\end{lemma}

\section{Lyapunov Functions and Convergence to the Stationary Distribution}

In this section, we shall find Lyapunov functions for Walsh diffusions to prove  existence and uniqueness of a stationary distribution, and convergence to this stationary distribution as $t \to \infty$. For general Markov processes, the application of Lyapunov functions has been widely studied in the last few decades. Without attempting to provide an exhaustive list of references, let us mention the following papers: \cite{BCG2008, DFG2009,  DMT1995, LMT1996, MT1993a, MT1993b, RT2000}. We have three goals: 

\smallskip

(a) Establish the very fact of long-term convergence of the transition kernel $P^t(x, \cdot)$ to the stationary distribution $\pi(\cdot)$ using Lyapunov functions, in a suitable distance;

\smallskip

(b) Prove that the rate of this convergence is exponential; that is, the distance between $P^t(x, \cdot)$ and $\pi(\cdot)$ is estimated from above as a constant (dependent on $x$) times $e^{-\vk t}$, for some $\vk > 0$;

\smallskip

(c) Estimate the rate $\vk$ of this exponential convergence.

\smallskip

\subsection{Definitions and general results} Let us start with general definitions. Consider an $\BR^d$-valued continuous-time Markov process $X = (X(t),\, t \ge 0)$ with transition kernel $P^t(x, \cdot)$. 

\begin{defn} We say that the process $X$ is {\it ergodic} if there exists a unique stationary distribution $\pi(\cdot)$, and if the transition kernel converges in the total variation norm $\, \lVert \cdot \rVert_{\text{TV}}\,$ to $\, \pi\,$, i.e., 
$$
\lim_{t\to \infty }\norm{P^t(x, \cdot) - \pi(\cdot)}_{\TV} \, =\,  0,\ \ \ \ \mbox{for every}\ \ x \in \BR^d.
$$
\end{defn}

\begin{defn} We say that $X$ is {\it $V$-uniformly ergodic} for a function $V : \BR^d \to [1, \infty)$, if $X$ is ergodic, and there exist constants $K, \vk > 0$ such that for every $t \ge 0,\, x \in \BR^d$, we have:
\begin{equation}
\label{eq:exp-ergodic}
\norm{P^t(x, \cdot) - \pi(\cdot)}_{\TV} \le K \, V(x)e^{-\vk t}.
\end{equation}
For $V \equiv 1$, we say that $X$ is {\it exponentially ergodic}. 
\end{defn}

\begin{asmp} \label{asmp:4} The spinning measure $\mu$ has finite support $\supp\mu = \{\theta_1, \ldots, \theta_p\}$ in $\mathbb S$; that is, the effective state space $\mathcal I^{\mu}$ in ~\eqref{eq:effective-state-space} for this Walsh diffusion is a finite union of rays from the origin 
$$
\mathcal I^{\mu} :=  \bigcup\limits_{i=1}^p\{\, r\theta_i \, \mid 0 \le r < l_i\},\ l_i := \ell(\theta_i),\ i = 1, \ldots, p.
$$
Sometimes, the Walsh diffusion in the finite union $\mathcal I^{\mu}$ of rays  is called a {\it spider}. %See also \cite[Example 10.1]{IchibaNew}.
\end{asmp}

%\begin{rmk} Under Assumption 4, the tree topology coincides with the Euclidean topology. 
%\end{rmk}

%Recall from Remark~\ref{rmk:effective-state-space} that the effective state space for the Walsh diffusion associated with $(g, \si, \mu)$ is given by~\eqref{eq:effective-state-space}. 

\begin{lemma} Under Assumptions \ref{asmp:1}, \ref{asmp:2}, and either \ref{asmp:3} or \ref{asmp:4},  we have the following results. 

\medskip

\noindent (a) Assume there exist a function $V : \mathcal I^{\mu} \to \BR_+$ in the domain of the generator $\CL$  and some positive constants $k, b, r_0$, such that %}, for some positive constants $k, b, r_0$, %we have:
\begin{equation}
\label{eq:Lyapunov-ergodicity}
\CL V(x) \le -k + b\, 1_{\mathcal B^{\mu}(r_0)}(x), \quad x \in \mathcal I^{\mu}, 
\end{equation}
%Let us introduce a new piece of notation 
\begin{equation}
\label{eq:ball}
\mbox{where}\ \ \mathcal B^{\mu}(r_0) := \{(r, \theta) \in \mathcal I^{\mu}\mid 0 < r \le r_0\}\cup\{\mathbf{0}\}.
\end{equation} 
Then the Walsh diffusion is ergodic. Moreover, for every bounded measurable function $f : \mathcal I^{\mu} \to \BR$, %we have:
\begin{equation}
\label{eq:running-average-convergence-non-R}
\lim\limits_{t \to \infty}\frac1t\int_0^tf(X(u))\,\md u = \int_{\mathcal I^{\mu}}f(x)\pi(\md x).
\end{equation}

\smallskip

\noindent (b) Assume there exist a function $V : \mathcal I^{\mu} \to [1, \infty)$ in the domain of the generator $\CL$ and some positive constants $k, b, r_0$, such that  %for some positive constants $k, b, r_0$, %we have:
\begin{equation}
\label{eq:Lyapunov-uniform-ergodicity}
\CL V(x) \le -kV(x) + b\, 1_{\mathcal B^{\mu}(r_0)}(x), \quad x \in \mathcal I^{\mu} .  
\end{equation}
Then the Walsh diffusion is $V$-uniformly ergodic, and $\int_{\mathcal I^{\mu}}V(x)\pi(\md x) < \infty$. 
\label{lemma:Lyapunov-exponent}
\end{lemma}

\begin{rmk} A function $V$ which satisfies~\eqref{eq:Lyapunov-ergodicity} or~\eqref{eq:Lyapunov-uniform-ergodicity} is called a {\it Lyapunov function} in the literature cited in the beginning of this section. 
\end{rmk}

\begin{proof} {\it Case 1.} First, we work under Assumptions \ref{asmp:1}, \ref{asmp:2}, \ref{asmp:3}. Then we operate in Euclidean topology. The set $\mathcal B^{\mu}(r_0)$ from~\eqref{eq:ball} is compact. The Walsh diffusion is Feller continuous from Lemma~\ref{lemma:Feller-Euclidean}, and has the positivity property from Theorem~\ref{thm:positivity}. 

\medskip

{\it Case 2.} Next, we work under Assumptions \ref{asmp:1}, \ref{asmp:2}, \ref{asmp:4}. Then we operate in the tree-topology. By Remark~\ref{rmk:compact-tree}, the set $\mathcal B^{\mu}(r_0)$ from~\eqref{eq:ball} is compact in the tree-topology. The Walsh diffusion is Feller continuous from Theorem~\ref{thm:Feller}, and has the positivity property from Theorem~\ref{thm:positivity}. 

\medskip

Finally, for both cases, the rest of the proof of (b) follows from \cite[Lemma 2.3, Theorem 2.6]{MyOwn10}, and the rest of the proof of (a) follows from \cite[Proposition 2.2]{MyOwn10} and \cite[Theorem 5.1]{MT1993b}. 
\end{proof}

\subsection{An example of convergence} Let us provide an example with explicit conditions on the drift and diffusion coefficients $g$ and $\si$.

\begin{lemma} Under Assumptions \ref{asmp:1}, \ref{asmp:2}, and either \ref{asmp:3} or \ref{asmp:4},  

\smallskip

\noindent (a) the Walsh diffusion is ergodic, if %If the following condition 
\begin{equation}
\label{eq:condition-on-g}
\varlimsup\limits_{r \to \infty}\sup\limits_{\theta \in \mathbb S}g(r, \theta) =: -\ol{g} < 0. 
\end{equation}
%is true, then the Walsh diffusion is ergodic. 

\smallskip

\noindent (b)  If, in addition, the following condition %both condition~\eqref{eq:condition-on-g} and the following condition 
\begin{equation}
\label{eq:condition-on-sigma}
\varlimsup\limits_{r \to \infty}\sup\limits_{\theta \in \mathbb S}\si(r, \theta) =: \ol{\si} < \infty,
\end{equation}
%are true, 
holds, then the Walsh diffusion is $V$-uniformly ergodic with $V = V_{\la}(r, \theta) := e^{\la r} $ for some $ \la > 0$.

\label{lemma:example}
\end{lemma}

\begin{proof} The proof is similar to that from \cite[Theorem 3.2]{MyOwn12}. Take a $C^{\infty}$ nondecreasing function $\phi : \BR_+ \to \BR_+$ such that 
$$
\phi(x) = 
\begin{cases}
0,\ \ x \le 1;\\
x,\ \ x \ge 2.
\end{cases}
$$
An example of such function can be found in \cite[section 3.2]{MyOwn12}. For (a), try $V(r, \theta) = \phi(r)$. This function satisfies condition~\eqref{eq:class-of-functions}, because $V'_r(0+, \theta) = \phi'(0) = 0$. It is also continuous in the Euclidean topology, which is the additional condition in Remark~\ref{rmk:generator-domain-Euclidean}. Plug into~\eqref{eq:explicit-generator} and get
\begin{equation}
\label{eq:phi-simplify}
\mbox{for}\ \ r \ge 2,\ \ \phi'(r) = 1,\ \ \mbox{and}\ \ \phi''(r) = 0 , 
\end{equation}
and hence, $\CL V(r, \theta) = g(r, \theta)$. It follows from~\eqref{eq:condition-on-g} that there exist $r_1, b > 0$ such that $g(r, \theta) \le -b$ for $r \ge r_1$. Therefore, for $r \ge r_0 := r_1\vee 2$, we get $\CL V(r, \theta) \le -b$. In addition, the function $\CL V$ is continuous and therefore bounded on $\mathcal B^{\mu}(r_0)$. Thus we have~\eqref{eq:Lyapunov-ergodicity}. Apply Lemma~\ref{lemma:Lyapunov-exponent} (a) to complete the proof of Lemma~\ref{lemma:example} (a). The proof of~\eqref{eq:running-average-convergence-non-R} follows from \cite[Theorem 8.1(a)]{MT1993a}. 

\smallskip

For (b), try $V(r, \theta) = \exp(\la\phi(r))$. Similarly, this function satisfies $V'_r(0+, \theta) = 0$ and therefore~\eqref{eq:class-of-functions}; and on top of this,  $V$ is also continuous in the Euclidean topology. Using~\eqref{eq:phi-simplify}, for $r \ge 2$, we have: $V'(r, \theta) = \la V(r, \theta)$, and $V''(r, \theta) = \la^2V(r, \theta)$. Thus 
\begin{equation}
\label{eq:generator-applied}
\CL V(r, \theta) = \left[g(r, \theta)\la + \frac12\si^2(r, \theta)\la^2\right]V(r, \theta)\ \ \mbox{for}\ \ r \ge 2.
\end{equation}
Now, from~\eqref{eq:condition-on-g} and~\eqref{eq:condition-on-sigma}, there exist constants $\ol{g}, \ol{\si} > 0$ and an $r_1 > 0$ such that $
g(r, \theta) \le -\ol{g} < 0,\ \ \si(r, \theta) \le \ol{\si},\ \ r \ge r_1.
$
Then, for $r \ge r_1$, $\la = \ol{g}\cdot\ol{\si}^{-2}$, we have 
\begin{equation}
\label{eq:coefficient-negative}
g(r, \theta)\la + \frac12\si^2(r, \theta)\la^2 \le -\ol{g}\la + \frac{\ol{\si}^2}2\la^2 \le -\frac{\ol{g}^2}{2\ol{\si}^2} =: -k < 0.
\end{equation}
Comparing~\eqref{eq:generator-applied} with~\eqref{eq:coefficient-negative}, we get $\CL V(r, \theta) \le -kV(r, \theta)$ for $\theta \in \supp\mu$ and  $r \ge r_1\vee 2 =: r_0$. Similarly to (a), we get that $\CL V$ is continuous and therefore bounded on $\mathcal B^{\mu}(r_0)$.  Therefore, we obtain~\eqref{eq:Lyapunov-uniform-ergodicity}. Apply Lemma~\ref{lemma:Lyapunov-exponent} (b) to complete the proof of Lemma~\ref{lemma:example} (b).
\end{proof}

\subsection{Explicit rate of exponential convergence} In a fairliy general setting, we can estimate the rate $\vk$ of exponential convergence from~\eqref{eq:exp-ergodic}. It is hard to estimate this rate for diffusions. Let us informally explain the difficulty: Let $\CL$ be the generator of a certain Markov process on $\BR^d$. Assume we have a Lyapunov function $V$ in the domain of this generator which satisfies
\begin{equation}
\label{eq:Lyapunov-general}
\CL V(x) \le -kV(x) + b\,1_C(x)\ \mbox{for all}\ x.
\end{equation}
Here, $k, b > 0$ are some constants, and $C$ is a ``small'' set. There is actually a precise meaning of the term {\it small set} in this theory, which was developed in \cite{DMT1995, MT1993a, MT1993b}. For our purposes, it is sufficient to let $C$ be a compact set, as follows from \cite[Lemma 2.3, Proposition 2.6]{MyOwn10}. One would like to infer an explicit value of the constant $\vk$ in~\eqref{eq:exp-ergodic} from the constants in~\eqref{eq:Lyapunov-general}. As mentioned in the Introduction, however, it turns out to be very hard, in general, see for example~\cite{Davies, Explicit, RR1996, RT2000}, since $\vk$ depends in a complicated way on $k, b, C$, and the transition kernel $P^t(x, \cdot)$. 

\smallskip

However, such estimates are much easier if the Markov process is on the half-line $\BR_+$, is stochastically ordered, and the ``exceptional set'' $C = \{0\}$ in the formula~\eqref{eq:Lyapunov-general} for a Lyapunov function. Alternatively, the stochastic process itself might not be stochastically ordered, but is stochastically dominated by a stochastically ordered Markov process with a Lyapunov function with $C = \{0\}$. This was done by  the coupling method in \cite{LMT1996} for some processes, including reflected diffusions, and in \cite{MyOwn12} for reflected jump-diffusions. 

\smallskip

Here, we are able to adjust the coupling techniques used in \cite{LMT1996, MyOwn12} for the case of Walsh diffusions, by dominating the radial component of the Walsh diffusion by a stochastically ordered reflected diffusion on $\BR_+$. The rest of this section closely follows the ideas of \cite{LMT1996, MyOwn12}. 

\smallskip

Recall $f : \BR_+ \to \BR$ is called {\it locally Lipschitz} if for every $R_0 > 0$ there exists a $C(R_0) > 0$ such that $
|f(r_1) - f(r_2)| \le C(R_0)|r_1 - r_2|,\ \ r_1, r_2 \in [0, C(R_0)].
$ Impose the following assumption.

\begin{asmp} The drift coefficient $g(r, \theta)$ is dominated by 
%an angular-independent drift coefficient $\ol{g}(r)$:
$
g(r, \theta) \le \ol{g}(r),\, (r, \theta) \in \mathcal I\setminus\{0\}, 
$
where $\ol{g}(r)$ is independent of $\theta$, and the diffusion coefficient $\si$ is itself angular-independent:
$$
\si(r, \theta) = \ol{\si}(r),\, (r, \theta) \in \mathcal I\setminus\{0\}.
$$
Here, $\ol{g}, \ol{\si} : \BR_+ \to \BR$ are assumed to be locally Lipschitz continuous, and so is $g(\cdot, \theta)$ for each $\theta$. 
%is also locally Lipschitz continuous. 
\label{asmp:domination}
\end{asmp}

\begin{thm} (a) Under Assumption~\ref{asmp:domination}, suppose there exists a constant $k > 0$ and a nondecreasing $C^2$ function $V : \BR_+ \to [1, \infty)$ such that 
\begin{equation}
\label{eq:Lyapunov-explicit}
\ol{g}(r)V'(r) + \frac12\ol{\si}^2(r)V''(r) \le - kV(r),\ r > 0.
\end{equation}
Then for every two points $x_1, x_2 \in \mathcal I$, and every $t > 0$, we have 
\begin{equation}
\label{eq:distance-V}
\norm{P^t(x_1, \cdot) - P^t(x_2, \cdot)}_V \le \left(V(x_1) + V(x_2)\right)e^{-kt}.
\end{equation}
%
%\smallskip
%
(b) If, in addition to (a), this Walsh diffusion is ergodic, and $\pi$ is its stationary distribution, satisfying $(\pi, V) < \infty$, then this Walsh diffusion is $V$-uniformly ergodic with $\vk := k$ in~\eqref{eq:exp-ergodic}. 
\label{thm:explicit-rate}
\end{thm}

\begin{rmk}
Somewhat abusing the notation, we will refer to $V$ sometimes as a function $V : \BR^d \to \BR$, and sometimes as a function $V : \BR_+ \to \BR$. 
\end{rmk}

\begin{proof} Part (a) {\it Step 1.} Similarly to \cite{LMT1996, MyOwn12}, we couple four processes: two copies $X_1$ and $X_2$ of the Walsh diffusion associated with $(g, \si, \mu)$, starting from $X_i(0) = x_i,\, i = 1, 2$, and two copies $S_1$ and $S_2$ of a reflected diffusion on $\BR_+$ with coefficients $\ol{g}(\cdot)$ and $\ol{\si}(\cdot)$, starting from 
%\begin{equation}
%\label{eq:starting-S}
$S_i(0) = \norm{x_i}$, $ i = 1, 2$,
%\end{equation}
so that the following pathwise comparison holds:
\begin{equation}
\label{eq:correct-coupling}
\norm{X_i(t)} \le S_i(t),\ i = 1, 2.
\end{equation}
Assume also $S_1$ and $S_2$ have the same driving Brownian motion. That is, there exists a standard Brownian motion $B = (B(t),\, t \ge 0)$, such that 
\begin{equation}
\label{eq:SDE-R}
\md S_i(t) = \ol{g}(S_i(t))\,\md t + \ol{\si}(S_i(t))\,\md B(t) + \md L_i(t),\, i = 1, 2.
\end{equation}
Here, $L_i = (L_i(t),\, t \ge 0),\, i = 1, 2$, are  some continuous adapted nondecreasing real-valued processes, with $L_i(0) = 0$, such that $L_i$ can increase only when $S_i = 0$, $i = 1, 2$. We shall show at the end of this proof how to construct such coupling. 

\smallskip

{\it Step 2.} Assume we have already constructed the coupling with all aforementioned properties. Then we can quickly prove the statement of Theorem~\ref{thm:explicit-rate}(a). Assume without loss of generality that $\norm{x_1} \le \norm{x_2}$. Then using the  standard comparison techniques for diffusions, we get 
%the following pathwise a.s. comparison:
\begin{equation}
\label{eq:comp-S}
S_1(t) \le S_2(t),\, t \ge 0.
\end{equation}

Define the stopping time $\tau := \inf\{t \ge 0\mid S_2(t) = 0\}.$
By~\eqref{eq:correct-coupling} and~\eqref{eq:comp-S}, we have  
$$
\norm{X_1(\tau)} = \norm{X_2(\tau)} = 0\ \Ra\ X_1(\tau) = X_2(\tau) = \mathbf{0}.
$$
Therefore, $\tau$ is a coupling time for $X_1$ and $X_2$. Using the standard trick, we assume $X_1(t) = X_2(t)$ a.s. for $t > \tau$. Then for a function $f : \mathcal I \to \BR$ such that $|f| \le V$, we have 
\begin{align}
\label{eq:main-523}
\begin{split}
& \left|\mathbb E [ f(X_1(t))]  - \mathbb E [f(X_2(t))] \right| = \left|\mathbb E \left[ f(X_1(t))1_{\{\tau > t\}}\right] - \mathbb E\left[ f(X_2(t))1_{\{\tau > t\}}\right]\right| \\ & \le \mathbb E[ |f(X_1(t))|1_{\{\tau > t\}}] + \mathbb E [|f(X_2(t))|1_{\{\tau > t\}}]  \le \mathbb E [V(X_1(t))1_{\{\tau > t\}} ] + \mathbb E [V(X_2(t))1_{\{\tau > t\}}] \\ & \le \mathbb E [V(S_1(t))1_{\{\tau > t\}}] + \mathbb E [V(S_2(t))1_{\{\tau > t\}}].
\end{split}
\end{align}

\smallskip

Combining~\eqref{eq:main-523} with~\eqref{eq:Lyapunov-explicit}, we get as in \cite{MyOwn12}:
$$
\left|\mathbb E [f(X_1(t))] - \mathbb E [f(X_2(t))] \right| \le \left(V(x_1) + V(x_2)\right)e^{-kt}.
$$
Taking supremum over all functions $f : \mathcal I \to \BR$ such that $|f| \le V$, we complete the proof of (a). 

\smallskip

{\it Step 3.} It remains to show that a coupling of $X_1, X_2, S_1, S_2$ which satisfies ~\eqref{eq:correct-coupling},~\eqref{eq:SDE-R} exists. First, our goal is to construct two Walsh diffusions $X_1$ and $X_2$ associated with $(g, \si, \mu)$, starting from $X_i(0) = x_i$, with the {\it same driving Brownian motion} $B$. That is, $X_1$ and $X_2$ need to satisfy %the equation
\begin{equation}
\label{eq:same-BM}
\md\norm{X_i(t)} = g(X_i(t))\,\md t + \si(X_i(t))\,\md B(t) + \md L^{\norm{X_i}}(t),\ i = 1, 2;\ t \ge 0.
\end{equation}
To this end, take a probability space $(\Oa, \CF, \MP)$ with infinitely many i.i.d. copies 
$$
W^{(n)} = (W^{(n)}(t),\, t \ge 0),\ \ n = 0, 1, 2, \ldots
$$
of a Walsh diffusion associated with $(g, \si, \mu)$, starting from the origin: $W^{(n)}(0) = \mathbf{0}$; as well as yet another independent standard Brownian motion $\ol{B} = (\ol{B}(t),\, t \ge 0)$. For each Walsh diffusion $W^{(n)}$, we can write a representation in terms of stochastic differential equation 
$$
\md\norm{W^{(n)}(t)} = g\left(W^{(n)}(t)\right)\,\md t + \si\left(W^{(n)}(t)\right)\,\md B_n(t) + \md L^{\norm{W_n}}(t),
$$
where $B_n = (B_n(t),\, t \ge 0),\, n = 1, 2, \ldots$ are i.i.d. standard Brownian motions. Let $x_i = \ol{r}_i\ol{\theta}_i,\, i = 1, 2$. If $\ol{r}_1 = 0$ or $\ol{r}_2 = 0$, then let $\tau_0 := 0$. Assume now $\ol{r}_1 > 0$ and $\ol{r}_2 > 0$. For $i = 1, 2$, consider strong solutions $\ol{S}_i$ of a one-dimensional SDE with drift coefficient $g(\cdot, \ol{\theta}_i)$ and diffusion coefficient $\si^2(\cdot, \ol{\theta}_i)$, starting from 
$\ol{S}_i(0) = \ol{r}_i$, driven by Brownian motion $\ol{B}$:
$$
\md \ol{S}_i(t) = g(\ol{S}_i, \ol{\theta}_i)\,\md t + \ol{\si}(\ol{S}_i)\,\md\ol{B}(t),\, t \le \ol{\tau}_i := \inf\{t \ge 0\mid \ol{S}_i(t) = 0\}.
$$
Because the drift coefficient $g(\cdot, \ol{\theta}_i)$ and the diffusion coefficient 
$\si(\cdot, \ol{\theta}_i)$ are locally Lipschitz continuous, this strong solution exists and is unique. By~\eqref{eq:Lyapunov-explicit} and Assumption~\ref{asmp:domination} it follows that this process is non-explosive, at least not until it hits zero. Define
$$
X_i(t) = \ol{\theta}_i \, \ol{S}_i(t),\ i = 1, 2;\ t < \tau_0;
$$
where the first stopping time is defined as $
\tau_0 :=  \inf\{t \ge 0\mid X_1(t) = \mathbf{0}\ \mbox{or}\ X_2(t) = \mathbf{0}\} = \ol{\tau}_1\wedge\ol{\tau}_2 $. 
Thus we defined $X_1(t)$ and $X_2(t)$ for $t \le \tau_0$ so that~\eqref{eq:same-BM} is satisfied with $B(t) := \ol{B}(t)$ for $t \le \tau_0$. 

%Without loss of generailty, let 
%$$
%X_1(\tau_0) = \mathbf{0},\ \ X_2(\tau_0) \ne \mathbf{0}.
%$$
%If, in fact, both $X_1(\tau_0) = X_2(\tau_0) = \mathbf{0}$, let $\tau_{\infty} := \tau_0$ and $X_i(t)$ for $t \ge \tau_{\infty}$, $i = 1, 2$, as at the end of this proof. 

\smallskip

{\it Step 4.} Next, we construct a sequence $(\tau_n)_{n \ge 0}$ of stopping times such that  $
\tau_0 \le \tau_1 \le \tau_2 \le \ldots $, 
and define $X_1$ and $X_2$ inductively on each $[\tau_{2k}, \tau_{2k+2}]$ so that~\eqref{eq:same-BM} holds with 
\begin{equation}
\label{eq:zeroes}
X_1(\tau_{2k}) = \mathbf{0},\ X_2(\tau_{2k+1}) = \mathbf{0},\ k \ge 0; \quad 
%\end{equation}
%Define the limit 
%\begin{equation}
%\label{eq:tau-limit}
\tau_{\infty} := \lim\limits_{k \to \infty}\tau_k\, . 
\end{equation}
Therefore, we shall construct $X_1$ and $X_2$ on $[0, \tau_{\infty}]$ so that~\eqref{eq:same-BM} and~\eqref{eq:zeroes} hold. In Step 6, we construct $X_1(t)$ and $X_2(t)$ for $t > \tau_{\infty}$, if $\tau_{\infty} < \infty$.

\smallskip

Assume we have already defined $X_1$ and $X_2$ on $[0, \tau_{2n}]$, so that~\eqref{eq:same-BM} and~\eqref{eq:zeroes} hold for $k = 0, \ldots, n$. In the case $
X_1(\tau_{2n}) = X_2(\tau_{2n}) = \mathbf{0}, $  we let $\tau_{k} := \tau_{2n}$ for $k \ge 2n$. Next, assume
$$
X_1\left(\tau_{2n}\right) = \mathbf{0},\ \ X_2\left(\tau_{2n}\right) \ne \mathbf{0}.
$$
%Define the next two stopping times:
%$$
%\tau_{2n+1} := \inf\{t \ge \tau_{2n}\mid X_2(t) = \mathbf{0}\};\ \tau_{2n+2} := \inf\{t \ge \tau_{2n}\mid X_1(t) = \mathbf{0}\},\ \mbox{for}\ n = 0, 1, 2, \ldots
%$$
Let $X_2(\tau_{2n}) = \rho_{2n}\theta_{2n}$. Construct a strong solution $S_{2n+1} = (S_{2n+1}(t),\, t \ge 0)$ of a one-dimensional SDE with coefficients $g(\cdot, \theta_{2n})$ and  $\si(\cdot)$, starting from $S_{2n+1}(0) = \rho_{2n}$, until it hits zero: 
$$
\md S_{2n+1}(t) = g(S_{2n+1}(t), \theta_{2n})\,\md t + \ol{\si}(S_{2n+1}(t))\,\md B_{2n+1}(t),\ t \le \tau'_{2n+1} := \inf\{t \ge 0\mid S_{2n+1}(t) = 0\}. 
$$
By local Lipschitz continuity of the coefficients $g(\theta_{2n}, \cdot)$ and $\si(\cdot)$, this strong solution exists and is unique. From~\eqref{eq:Lyapunov-explicit} and Assumption~\ref{asmp:domination}, it follows that this process is non-explosive, at least not until it hits zero. Define $\tau_{2n+1} := \tau_{2n} + \tau'_{2n+1}$, and 
$$
X_1\left(t + \tau_{2n}\right) := W^{(2n+1)}(t),\, \quad  X_2\left(t + \tau_{2n}\right) := S_{2n+1}(t),\, \quad t \le \tau'_{2n+1}.
$$
This defines $X_1$ and $X_2$ on the next time interval $[\tau_{2n}, \tau_{2n+1}]$ such that~\eqref{eq:same-BM} is satisfied. The common driving Brownian motion $B$ is defined as
$$
B(t) := B\left(\tau_{2n}\right) + B_{2n+1}\left(t - \tau_{2n}\right),\, t \in [\tau_{2n}, \tau_{2n+1}].
$$
If we have $ X_1\left(\tau_{2n+1}\right) = X_2\left(\tau_{2n+1}\right) = \mathbf{0} $, we let $\tau_{k} := \tau_{2n+1}$ for $k \ge 2n+1$. Otherwise, we repeat the construction above with $X_1$ and $X_2$ swapped, with obvious changes. This allows us to construct $X_1$ and $X_2$ on $[\tau_{2n+1}, \tau_{2n+2}]$ so that~\eqref{eq:same-BM} and~\eqref{eq:zeroes} hold for $k \le 2n+2$. The common driving Brownian motion $B$ is defined as
$$
B(t) := B\left(\tau_{2n+1}\right) + B_{2n+2}\left(t - \tau_{2n+1}\right),\ t \in [\tau_{2n+1}, \tau_{2n+2}].
$$
By induction, we have defined $X_1(t)$ and $X_2(t)$ for $t \le \tau_{\infty}$.

\smallskip

{\it Step 5.} It follows from~\eqref{eq:zeroes} that $X_1\left(\tau_{\infty}\right) = X_2\left(\tau_{\infty}\right) = \mathbf{0}  $. It suffices to define $\,X_{i}\,$, $\,i =1, 2\,$ on $\,[\tau_{\infty}, \infty)\,$ if $\,\tau_{\infty} < \infty\,$. 
%it suffices to show $\tau_{\infty} = \infty$. 
Assume $\tau_{\infty} < \infty$ and 
%From~\eqref{eq:zeroes}, it follows that  $X_1\left(\tau_{\infty}\right) = X_2\left(\tau_{\infty}\right) = \mathbf{0}.  $ 
define $X_1(t)$ and $X_2(t)$ for $t \ge \tau_{\infty}$ by %as follows: 
$
X_1\left(t + \tau_{\infty}\right) = X_2\left(t + \tau_{\infty}\right) = W^{(0)}(t)$, $t \ge 0$.
The common driving Brownian motion $B$ is defined on $[\tau_{\infty}, \infty)$ as
$$
B(t) := B\left(\tau_{\infty}\right) + B_0\left(t - \tau_{\infty}\right).
$$
This completes the construction of two copies $X_1$ and $X_2$ of the Walsh diffusion, associated with $(g, \si, \mu)$, starting from $X_i(0) = x_i$, $i = 1, 2$, such that~\eqref{eq:same-BM} holds. 

\smallskip

{\it Step 6.} Finally, we construct strong versions $S_1$ and $S_2$ of a reflected diffusion with drift coefficient $\ol{g}$ and diffusion coefficient $\ol{\si}$, starting from $S_i(0) = \ol{r}_i$, with driving Brownian motion $B$, as in~\eqref{eq:SDE-R}. This is possible by classic results, because $\ol{g}$ and $\ol{\si}$ are locally Lipschitz continuous. By standard comparison techniques, we have~\eqref{eq:correct-coupling}. 

\smallskip

Part (b) Take a function $f : \mathcal I \to \BR$ with $|f| \le V$. Then %the Walsh diffusion 
$X = (X(t),\, t \ge 0)$ satisfies
$$
\left|\mathbb E_{x_1}[f(X(t))] - \mathbb E_{x_2}[f(X(t))]\right| \le \left(V(x_1) + V(x_2)\right)e^{-kt},\ x_1, x_2 \ge 0.
$$
Integrate over $\mathcal I$ with respect to $x_2 \sim \pi$. Then we have:
\begin{align*}
\left|\mathbb E_{x_1}[f(X(t))] - (\pi, f)\right| &= \left| \mathbb E_{x_1}[ f(X(t)) ] - \int_{\mathcal I}\mathbb E_{x_2}[f(X(t))]\,\pi(\md x_2)\right| \\ & \hspace{-1cm}  
\le \int_{\mathcal I}\left|\mathbb E_{x_1}[f(X(t))] - \mathbb E_{x_2}[f(X(t))]\right|\,\pi(\md x_2) \le e^{-kt}\left(V(x_1) + \int_{\mathcal I}V(x_2)\pi(\md x_2)\right) \\ & \le e^{-kt}\left(V(x_1) + (\pi, V)\right) \le e^{-kt}(1 + (\pi, V))V(x_1). 
\end{align*}
\end{proof}

\begin{rmk}
In Assumption~\ref{asmp:domination}, we imposed Lipschitz continuity assumption only to guarantee strong existence and pathwise uniqueness. It is well known, however, that these existence and uniqueness results hold under weaker conditions. In this case, our result also holds. 
\end{rmk}

Let us present some corollaries. Under assumptions of Theorem~\ref{thm:explicit-rate}, define 
\begin{equation}
\label{eq:K-explicit}
K(x, \la) := \ol{g}(x)\la + \ol{\si}^2(x)\frac{\la^2}2.
\end{equation}
The following result is proved similarly to \cite[Theorem 4.3, Corollary 5.2]{MyOwn12}.

\begin{cor} Under Assumption~\ref{asmp:domination}, suppose there exist $k,\, \la > 0$ such that 
\begin{equation}
\label{eq:K}
\sup\limits_{x > 0}K(x, \la) =: -k < 0.
\end{equation}
Then the Walsh diffusion $X$ is $V(r) := e^{\la r}$-uniformly ergodic with $\vk := k$, and the stationary distribution $\pi$ satisfies $(\pi, V) < \infty$. 
\end{cor}

\begin{rmk} Under Assumption~\ref{asmp:domination}, suppose 
$$
\sup\limits_{x \in \mathcal I\setminus\{0\}}\ol{g}(x) =: -\tilde{g} < 0,\ \ \sup\limits_{x \in \mathcal I\setminus\{0\}}\ol{\si}(x) =: \tilde{\si} < \infty.
$$
Similarly to \cite[Corollary 4.4]{MyOwn12}, we can show that~\eqref{eq:K} holds with $\la := \tilde{g}/\tilde{\si}$. Then, the Walsh diffusion $X$ is $V(r) := e^{\la r}$-uniformly ergodic with rate $\vk := k$ of exponential convergence from~\eqref{eq:K}. 
\end{rmk}

\begin{exm} [Continuing from Example \ref{ex:1}] Consider a Walsh diffusion $X$ associated with $(g, \si, \mu)$, such that the drift and diffusion coefficients are constant:
$$
g(r, \theta) \equiv g < 0,\ \ \si(r, \theta) = \si > 0.
$$
Then we can take $\ol{g}(r) := g$, and $\ol{\si}(r) := \si$ in the Assumption~\ref{asmp:domination}. The expression $K$ in ~\eqref{eq:K-explicit} is reduced to %independent of $x$:
$K(x, \la) = g\la + \frac{\si^2}2\la^2,$
and is minimized at $\la = - g/\si^{2}$, with minimum $-k$, where 
\begin{equation}
\label{eq:explicit-const}
k := \frac{g^2}{2\si^2}.
\end{equation}
This is an estimate of the exponential rate $\vk$ of convergence. Actually, this estimate is exact. Indeed, the radial component $\norm{X(\cdot)}$ of such Walsh diffusion is a reflected Brownian motion on the half-line with drift $g$ and diffusion $\si^2$, and such process is known from \cite{LMT1996} to have exact rate of exponential convergence as in~\eqref{eq:explicit-const}. 
\end{exm}

We can apply this result to (non-reflected) diffusions on the real line. 

\begin{cor} \label{cor:DiffusionR} 
Take a solution $X = (X(t),\, t \ge 0)$ of an SDE on the real line $\BR$ with drift coefficient $g$ and diffusion coefficient $\si^2$.
Suppose it is ergodic. Assume that  $g$ and $\si$ are locally Lipschitz, and $\si(x) = \si(-x)$ for all $x \in \BR$. Suppose there exists a $C^2$ nondecreasing function $V : \BR_+ \to [1, \infty)$ and a constant $k > 0$ such that 
$$
\ol{g}(r)V'(r) + \frac12\si^2(r)V''(r) \le -kV(r),\ \mbox{where}\, \, \,  \ol{g}(r) := g(r)\vee(-g(-r)),\ r > 0;
$$
Finally, assume that the stationary distribution $\pi$ satisfies $(\pi, V) < \infty$. Then $X$ is $V$-uniformly ergodic with rate of convergence $\vk = k$. 
\end{cor}

\begin{proof} Follows from Theorem~\ref{thm:explicit-rate} and the observation that the real line can be thought of as a ``spider" with two rays, corresponding to North and South Pole, i.e., $\theta_1$ and $\theta_2$. Then the process $X$ becomes a spider associated with $(\tilde{g}, \si, \mu)$, where $\mu$ is a uniform measure on $\{\theta_1, \theta_2\}$, and 
$ \tilde{g}(r, \theta_i) := g(r) \cdot 1_{\{i =1 \}} - g(-r) \cdot 1_{\{i = 2\}} $, $r \ge 0$. 
%$$
%\tilde{g}(r, \theta_i) = 
%\begin{cases}
%g(r),\ i = 1;\\
%-g(-r),\ i = 2.
%\end{cases}
%$$
\end{proof}

\begin{exm}[Bang-bang drifts \cite{KaratzasShreve}]
Consider a diffusion on $\mathbb R$ %the real line 
with drift and diffusion coefficients 
$ g(x) = - g_1 \cdot 1_{\{ x > 0\}} + g_2 \cdot 1_{\{x \le 0\}}$, with constants $g_1, g_2 > 0$, and $\sigma (x) \equiv 1$, $x \in \mathbb R$. 
%$$
%g(x) = 
%\begin{cases}
%2,\ x > 0;\\
%-1, x \le 0;
%\end{cases}
%\ \ \mbox{and}\ \ \si(x) \equiv 1.
%$$
Then we can take $\ol{g}(r) = g_1\wedge g_2$, and by Corollary \ref{cor:DiffusionR} it is $V$-uniformly ergodic with $\lambda = g_1\wedge g_2$ and rate  $\vk = k = (g_1\wedge g_2)^2/2$ of exponential convergence. 
%hence, we can take $\la = 1$ and $\vk = k = 0.5$ in~\eqref{eq:K}. 
\end{exm}

\section{Appendix}

\begin{lemma} \label{lemma:piecewise-functions}
Take two sequences $(f_n)_{n \ge 0}$ and $(g_n)_{n \ge 0}$ of continuous functions $[0, T] \to \BR^d$. Assume $t_n \in [0, T]$ are such that $f_n(t_n) = g_n(0) = \mathbf{0}$ for $n = 0, 1, 2, \ldots$ and $t_n \to t_0$. Assume $f_n \to f_0$ and $g_n \to g_0$ uniformly on $[0, T]$. Define the new sequence of functions $h_n : [0, T] \to \BR^d$:
$$
h_n(t) := 
\begin{cases}
f_n(t),\ t \in [0, t_n];\\
g_n(t - t_n),\ t \in [t_n, T].
\end{cases}
$$
Then $h_n \to h_0$ uniformly on $[0, T]$. 
\end{lemma}

\begin{proof} Without loss of generality, assume $t_n \uparrow t_0$, as $\,n \to \infty\,$. Otherwise, we can switch from $t_n$ to $T - t_n$ and from $h_n(t)$ to $h_n(T - t)$. Then we can estimate the maximum difference 
\begin{align*}
&\max\limits_{t \in [0, T]}\norm{h_n(t) - h_0(t)} \\ & \le \max\limits_{t \in [0, t_n]}\norm{h_n(t) - h_0(t)} + \max\limits_{t \in [t_n, t_0]}\norm{h_n(t) - h_0(t)}  + \max\limits_{t \in [t_0, T]}\norm{h_n(t) - h_0(t)} \\ & = \max\limits_{t \in [0, t_n]}\norm{f_n(t) - f_0(t)} + \max\limits_{t \in [t_n, t_0]}\norm{g_n(t - t_n) - f_0(t)} +  
\max\limits_{t \in [t_0, T]}\norm{g_n(t - t_n) - g_0(t - t_0)}.
\end{align*}

By uniform convergence $f_n \to f_0$, the first term can be estimated as
\begin{equation}
\label{eq:term-1}
\max\limits_{t \in [0, t_n]}\norm{f_n(t) - f_0(t)} \le \max\limits_{t \in [0, t_0]}\norm{f_n(t) - f_0(t)} \to 0.
\end{equation}
The second term can be estimated as
\begin{equation}
\label{eq:term-2}
\max\limits_{t \in [t_n, t_0]}\norm{g_n(t - t_n) - f_0(t)} \le \max\limits_{s \in [0, t_0 - t_n]}\norm{g_n(s)} + \max\limits_{t \in [t_n, t_0]}\norm{f_0(t)}.
\end{equation}
Note that $\norm{g_n} \to \norm{g_0}$ uniformly on $[0, T]$, and $t_0 - t_n \to 0$ as $n \to \infty$. Thus, 
\begin{equation}
\label{eq:term-21}
\max\limits_{s \in [0, t_0 - t_n]}\norm{g_n(s)} \to \norm{g_0(0)} = 0.
\end{equation}
Since $f_0(t_0) = \mathbf{0}$, $t_n \to t_0$, and $\norm{f_0}$ is continuous, the second term in the right-hand side of~\eqref{eq:term-2} converges to zero, as $\,n \to \infty\,$. Combining this with~\eqref{eq:term-21}, we get  
\begin{equation}
\label{eq:term-2-to-0}
\max\limits_{t \in [t_n, t_0]}\norm{g_n(t - t_n) - f_0(t)} \to 0.
\end{equation}

Thirdly, the third term can be estimated by 
\begin{equation}
\label{eq:term-3}
\max\limits_{t \in [t_0, T]}\norm{g_n(t - t_n) - g_0(t - t_0)} \le \max\limits_{t \in [t_0, T]}\norm{g_n(t - t_n) - g_0(t - t_n)} + \max\limits_{t \in [t_0, T]}\norm{g_0(t - t_n) - g_0(t - t_0)}.
\end{equation}
The first term in the right-hand side of~\eqref{eq:term-3} can be estimated as
\begin{equation}
\label{eq:term-31}
\max\limits_{t \in [t_0, T]}\norm{g_n(t - t_n) - g_0(t - t_n)}  \le \max\limits_{s \in [0, T]}\norm{g_n(s) - g_0(s)} \to 0.
\end{equation}
The second term in the right-hand side of~\eqref{eq:term-3} also tends to zero as $n \to \infty$, because $g_0$ is continuous, and therefore uniformly continuous on $[0, T]$, while $t_n \to t_0$. Combining this observation with~\eqref{eq:term-31}, we get that 
\begin{equation}
\label{eq:term-3-to-0}
\max\limits_{t \in [t_0, T]}\norm{g_n(t - t_n) - g_0(t - t_0)} \to 0.
\end{equation}

Finally, combining~\eqref{eq:term-1},~\eqref{eq:term-2-to-0},~\eqref{eq:term-3-to-0}, we complete the proof of Lemma~\ref{lemma:piecewise-functions}. 
\end{proof}

%%%%%%%%%%%%%%%%%%

\begin{lemma} For a Walsh diffusion $X$ without drift from Case 3 of the proof of Theorem~\ref{thm:positivity}, we have $\mathbb P\left(\tau^X_{R, A} < t\right) > 0$, where $\tau^{X}_{R, A}$ is defined in \eqref{eq:stopping-times-1}.
\label{lemma:intermediate-positivity}
\end{lemma}

\begin{proof}[Proof of Lemma \ref{lemma:intermediate-positivity}] 
%\footnote{(TI) Shall we move this proof of Lemma from here to Appendix? (AS) Let us leave it here, let us not make an Appendix.}
By Assumption \ref{asmp:1}, the functions $\si$ and $\si^{-1}$ are bounded on the (open) ball $D$ with radius $R$, which, together with its closure $\ol{D}$, lies inside $I$. This implies 
\begin{equation}
\label{eq:estimates-on-sigma}
0 < c_1 \le \si^2(x) \le c_2 < \infty\ \mbox{for all}\ \ x \in \ol{D}.
\end{equation}
Next, we can estimate the probability from Lemma~\ref{lemma:intermediate-positivity} as follows: 
\begin{equation}
\label{eq:first-estimate-of-hitting-time}
\mathbb P\left(\tau^X_{R, A} < t\right) \ge \mathbb P\left(\tau^X_{R, A} < t,\, \tau^X_{R^\prime} > t\right).
\end{equation}
It follows from~\eqref{eq:estimates-on-sigma} and~\eqref{eq:expression-for-time-change} that 
\begin{equation}
\label{eq:growth-of-time-change}
c_1t \le T(t) \le c_2t,\ \ \mbox{for}\ \ t \in [0, \tau^X_{R^\prime}].
\end{equation}
Also, from definitions of hitting times~\eqref{eq:stopping-times-1} -
%~\eqref{eq:stopping-times-2},
~\eqref{eq:stopping-times-3}, %~\eqref{eq:stopping-times-4}, 
it immediately follows that
\begin{equation}
\label{eq:time-change-and-hitting-times}
T\left(\tau^X_{R'}\right) = \tau^W_{R'},\ \ T\left(\tau^X_{R, A}\right) = \tau^W_{R, A}.
\end{equation}
The estimate~\eqref{eq:growth-of-time-change} and the relations~\eqref{eq:time-change-and-hitting-times}, in turn, imply that the inverted time-change $T^{-1}(s) := \inf\{t \ge 0\mid T(t) \ge s\}$ satisfies
\begin{equation}
\label{eq:growth-of-inverted-time-change}
c_2^{-1}s \le T^{-1}(s) \le c_1^{-1}s,\ \ \mbox{for}\ \ s \in [0, \tau^W_{R^\prime}].
\end{equation}

Let us show that 
\begin{equation}
\label{eq:comparison-of-events}
\left\{\tau^W_{R, A} < c_1t < c_2t < \tau^W_{R^\prime}\right\} \subseteq \left\{\tau^X_{R, A} < t < \tau^X_{R^\prime}\right\}.
\end{equation}
Indeed, assume the event in the left-hand side of~\eqref{eq:comparison-of-events} has happened. Then from~\eqref{eq:time-change-and-hitting-times}, applying the inverted time-change $T^{-1}$, we get 
\begin{equation}
\label{eq:new-event}
\tau^X_{R, A} < T^{-1}(c_1t) < T^{-1}(c_2t) < \tau^X_{R^\prime}.
\end{equation}
Applying~\eqref{eq:growth-of-inverted-time-change} to~\eqref{eq:new-event}, we have 
%\begin{equation}
%\label{eq:estimates-for-T}
$ T^{-1}(c_1t) \le t,\ \ T^{-1}(c_2t) \ge t$.
%\end{equation}
Comparing this %~\eqref{eq:estimates-for-T} 
with~\eqref{eq:new-event}, we get $\tau^X_{R, A} < t < \tau^X_{R^\prime}$. This completes the proof of~\eqref{eq:comparison-of-events}. 

To complete the proof of Lemma~\ref{lemma:intermediate-positivity}, we need only to show that 
\begin{equation}
\label{eq:new-event-positive}
\mathbb P\left(\tau^W_{R, A} < c_1t < c_2t < \tau^W_{R'}\right) > 0.
\end{equation}
Indeed, $W$ is Walsh Brownian motion, starting from the origin. As noted earlier before, $\norm{W(t)} = Z(t)$ is a reflected Brownian motion on the half-line, starting from zero. Define $\tau^Z_a := \inf\{t \ge 0\mid Z(t) = a\}$ for $a > 0$. Then $\arg W(\tau^Z_a) \sim \mu$ is independent of $Z$. Therefore, the probability in the left-hand side of~\eqref{eq:new-event-positive} is equal to 
\begin{equation}
\label{eq:positive}
\mu(A)\cdot\mathbb P\left(\tau^Z_R < c_1t < c_2t < \tau^Z_{R'}\right) > 0.
\end{equation}
That the left-hand side of~\eqref{eq:positive} is indeed positive follows from $\mu(A) > 0$ and the properties of a reflected Brownian motion on the half-line. This proves~\eqref{eq:new-event-positive}, and with it Lemma~\ref{lemma:intermediate-positivity}. 
\end{proof} 

\medskip

%%%%%%%%%%%%%%%%%%%%%

\section*{Acknowledgements} The authors are thankful to \textsc{Krzysztof Burdzy, Zhen-Qing Chen, Ioannis Karatzas, Minghan Yan,  Kouji Yano} for useful discussions. They would also like to express their gratitude to \textsc{Zhen-Qing Chen} for the Dirichlet form construction. They are indebted to anonymous referees and associate editor who pointed out  errors in the earlier manuscripts and made numerous suggestions. 
%of this process. 
The first author is supposed in part by NSF grants DMS 1313373 and 1615229. The second author is supported in part by NSF grant DMS 1409434 of \textsc{Jean-Pierre Fouque}.

%%%%%%%%%%%%%%%%%%%%%%%%%%%%
%%%%%%%%%%%%%%%%%%%%%%%

\end{document}